\documentclass[12pt]{amsart}
\usepackage{a4wide,enumerate,xcolor}
\usepackage{amsmath}
\allowdisplaybreaks
\usepackage{amssymb,amsthm,bm}

\let\pa\partial
\let\na\nabla
\let\eps\varepsilon
\newcommand{\N}{{\mathbb N}}
\newcommand{\R}{{\mathbb R}}
\newcommand{\diver}{\operatorname{div}}

\newcommand{\dx}{\mathrm{d}x}
\newcommand{\ds}{\mathrm{d}s}
\newcommand{\dt}{\mathrm{d}t}

\newtheorem{theorem}{Theorem}
\newtheorem{lemma}[theorem]{Lemma}

\newtheorem{remark}[theorem]{Remark}
\newtheorem{corollary}[theorem]{Corollary}

\begin{document}

\title[Nonisothermal Maxwell--Stefan systems]{
	Global existence of weak solutions and  \\
	weak--strong uniqueness for \\
	nonisothermal Maxwell--Stefan systems}

\author[S. Georgiadis]{Stefanos Georgiadis}
\address{Computer, Electrical and Mathematical Science and Engineering Division, King
Abdullah University of Science and Technology (KAUST), Thuwal 23955-6900, Saudi-Arabia}
\email{stefanos.georgiadis@kaust.edu.sa}

\author[A. J\"ungel]{Ansgar J\"ungel}
\address{Institute of Analysis and Scientific Computing, Technische Universit\"at Wien,
Wiedner Hauptstra\ss e 8--10, 1040 Wien, Austria}
\email{juengel@tuwien.ac.at}

\date{}

\thanks{The authors acknowledge partial support from
the Austrian Science Fund (FWF), grants P33010 and F65.
This work has received funding from the European
Research Council (ERC) under the European Union's Horizon 2020 research and
innovation programme, ERC Advanced Grant no.~101018153.}

\begin{abstract}
The dynamics of multicomponent gas mixtures with vanishing barycentric velocity
is described by Maxwell--Stefan equations with mass diffusion and heat conduction.
The equations consist of the mass and energy balances, coupled to an algebraic system
that relates the partial velocities and driving forces.
The global existence of weak solutions to this system in a bounded domain with
no-flux boundary conditions is proved by using the boundedness-by-entropy method.
A priori estimates are obtained from the entropy inequality which originates from
the consistent thermodynamic modeling. Furthermore, the weak--strong uniqueness property
is shown by using the relative entropy method.
\end{abstract}



\maketitle


\section{Introduction}

The dynamics of multicomponent gaseous mixtures with vanishing barycentric velocity
and constant temperature can be described by the Maxwell--Stefan equations \cite{Max66,Ste71}. 
The existence of local-in-time smooth and
global-in-time weak solutions to these systems 
has been proved in \cite{Bot11,GiMa98,HMPW17,JuSt13}.
The analysis of {\em nonisothermal} gas mixtures is, however, incomplete.
The existence of local-in-time solutions was shown in \cite{HuSa18}, while
\cite{HeJu21} investigated a special nonisothermal case. In this paper, 
we prove the existence of global-in-time weak solutions and the weak--strong
uniqueness property for a rather general nonisothermal Maxwell--Stefan
system. The novelty of our approach is the consistent thermodynamic modeling.

\subsection{Model equations}

The evolution of the mass densities $\rho_i(x,t)$ of the $i$th gas component and the
temperature $\theta(x,t)$ of the mixture is described by the mass and energy 
balances 
\begin{align}
  & \pa_t\rho_i + \diver J_i = 0, \quad \pa_t(\rho e) + \diver J_e = 0, \quad i=1,\ldots,n, 
	\label{1.mass} \\
	& J_i = \rho_i u_i, \quad J_e = -\kappa(\theta)\na\theta
	+ \sum_{j=1}^n(\rho_j e_j + p_j)u_j\quad\mbox{in }\Omega,\ t>0, \label{1.energy}
\end{align}
where $\Omega\subset\R^3$ is a bounded Lipschitz domain,
$J_i$ and $J_e$ are the diffusion and energy fluxes, respectively,
$u_i$ are the diffusional velocities, $\rho=\sum_{i=1}^n\rho_i$ is the total mass density,
$p_i$ the partial pressure with the total pressure $p=\sum_{i=1}^n p_i$, 
$\rho_i e_i$ the partial internal energy $\rho_ie_i$ with the total energy $\rho e=\sum_{i=1}^n
\rho_i e_i$, and $\kappa(\theta)$ is the heat conductivity.
Equations \eqref{1.mass}--\eqref{1.energy} are supplemented with the boundary and initial 
conditions
\begin{align}
  J_i\cdot\nu = 0, \quad J_e\cdot\nu = \lambda(\theta-\theta_0) 
	&\quad\mbox{on }\pa\Omega,\ t>0, \label{1.bc} \\
  \rho_i(0) = \rho_i^0, \quad \theta(0)=\theta^0 &\quad\mbox{in }\Omega,\
	i=1,\ldots,n, \label{1.ic}
\end{align}
where $\nu$ is the exterior unit normal vector to $\pa\Omega$, $\theta_0>0$ is the given
background temperature, and $\lambda>0$ is a relaxation constant. The boundary conditions
mean that the gas components cannot leave the domain, while heat exchange through
the boundary is possible and proportional to the difference between the gas
and background temperatures. To close the model, we need to
determine $u_i$, $\rho_ie_i$, and $p_i$.
 
The velocities $u_i$ are computed from the constrained algebraic Maxwell--Stefan system
\begin{equation}\label{1.velo}
  -\theta\sum_{j=1}^n b_{ij}\rho_i\rho_j(u_i-u_j) = d_i \quad\mbox{for }
	i=1,\ldots,n, \quad \sum_{i=1}^n\rho_i u_i=0,
\end{equation}
where the constant coefficients $b_{ij}=b_{ji}>0$ model the interaction between the 
$i$th and $j$th components. The driving force $d_i$ is given by 
\begin{equation}\label{1.driving}
  d_i = \rho_i\theta\na\frac{\mu_i}{\theta} - \theta(\rho_i e_i+p_i)\na\frac{1}{\theta},
	\quad i=1,\ldots,n,
\end{equation}
where $\mu_i$ is the chemical potential. The constraint
\begin{equation}\label{1.pressure}
  \na p = 0\quad\mbox{in }\Omega,\ t>0,
\end{equation}
is needed in order for our system to be thermodynamically consistent.
We refer to Section \ref{sec.model} for details.

The internal energies $\rho_ie_i$ and chemical potentials $\mu_i$
are determined from the Helmholtz free energy (see \eqref{2.helm}), 
and the pressure is computed from the Gibbs--Duhem relation.
As shown in Section \ref{sec.model}, these quantities are explicitly given by
\begin{equation}\label{1.constit}
\begin{aligned}
  \mu_i &= \frac{\theta}{m_i}\log\frac{\rho_i}{m_i} - c_w\theta(\log\theta-1), 
	&\quad \rho_ie_i &= c_w\rho_i\theta, \\
	\rho_i\eta_i &= -\frac{\rho_i}{m_i}\bigg(\log\frac{\rho_i}{m_i}-1\bigg)
	+ c_w\rho_i\log\theta, 
	&\quad p_i &= \frac{\rho_i\theta}{m_i}, \quad i=1,\ldots,n,
\end{aligned}
\end{equation}
where $\rho_i\eta_i$ is the entropy density of the $i$th component and $c_w>0$ is the
heat capacity. Then the driving force $d_i$ and energy flux $J_e$ simplify to
\begin{equation}\label{1.Je}
  d_i = \frac{\na(\rho_i\theta)}{m_i}, \quad
  J_e = -\kappa\na\theta + \theta\sum_{i=1}^n\frac{\rho_i u_i}{m_i}.
\end{equation}

The matrix $M$ associated to the algebraic system \eqref{1.velo} is singular 
(since $\sum_{i=1}^n d_i=0$) and thus
not positive definite. However, we recall in Section \ref{sec.inv} that it is positive
definite on the subspace $L=\{\bm{y}=(y_1,\ldots,y_n)\in\R^n:\sqrt{\bm{\rho}}\cdot\bm{y}=0\}$
(here, $\sqrt{\bm{\rho}}$ is the vector with components $\sqrt{\rho_i}$).
Therefore, the Bott--Duffin inverse of $M$, denoted by $M^{BD}=M^{BD}(\bm\rho)$, 
exists and is symmetric and positive definite on $L$.
Moreover, we show in Section \ref{sec.flux} below that the fluxes can be expressed as
a linear combination of the entropy variables (or thermo-chemical potentials)
$\bm\mu/\theta=(\mu_1/\theta,\ldots,\mu_n/\theta)$ and $-1/\theta$,
\begin{equation}\label{1.onsager}
  \begin{pmatrix} \bm{J} \\ J_e \end{pmatrix}
	= -Q(\bm\rho,\theta)\begin{pmatrix} \bm{\mu}/\theta \\ -1/\theta \end{pmatrix},
	\quad\mbox{where }
	Q(\bm\rho,\theta) = \begin{pmatrix}
	A & \bm{B} \\ \bm{B}^T & a \end{pmatrix},
\end{equation}
and $A=(A_{ij})\in\R^{n\times n}$, $\bm{B}=(B_i)\in\R^n$, $a>0$ are given by
\begin{equation}\label{1.AB}
  A_{ij}(\bm\rho) = M_{ij}^{BD}\sqrt{\rho_i\rho_j}, \quad
	B_i(\bm\rho,\theta) = \theta\sum_{j=1}^n\frac{A_{ij}}{m_j}, \quad
	a(\bm\rho,\theta) = \theta^2\bigg(\kappa + \sum_{i,j=1}^n\frac{A_{ij}}{m_im_j}\bigg).
\end{equation}
Here, variables in bold font are $n$-dimensional vectors. 
The Onsager matrix $Q$ turns out to be positive semidefinite (see \eqref{3.semidefQ}),
which reveals the parabolic structure of equations \eqref{1.mass}--\eqref{1.energy}.

\subsection{State of the art}

The isothermal Maxwell--Stefan equations can be derived from the multispecies 
Boltzmann equations in the diffusive approximation \cite{BGPS13}. The high-friction
limit in Euler (--Korteweg) equations reveals a formal gradient-flow form of the
Maxwell--Stefan equations \cite{HJT19}, leading to Fick--Onsager diffusion fluxes
instead of \eqref{1.velo}. In fact, it is shown in \cite{BoDr23} that the
Fick--Onsager and generalized Maxwell--Stefan approaches are equivalent.
A formal Chapman--Enskog expansion of the stationary nonisothermal model was given
in \cite{TaAo99}. Another nonisothermal Maxwell--Stefan system was derived in \cite{ABSS20}, 
but with a different energy flux than ours.

Maxwell--Stefan systems with nonvanishing barycentric velocities can be formulated
in the framework of hyperbolic--parabolic systems, 
which allows one to perform a local-in-time
existence analysis \cite{GiMa98}. Global-in-time regular solutions around the
constant equilibrium state were found to exist in \cite{GiMa98a}.
An existence analysis for Maxwell-Stefan systems coupled to the Navier--Stokes equations
for the barycentric velocity can be found in \cite{ChJu15} for the incompressible case
and in \cite{BoDr21} for the compressible situation.
For steady-state problems, we refer to, e.g., \cite{BJPZ22,MPZ15}. 

When the barycentric velocity vanishes, the (isothermal) Maxwell--Stefan equations can be solved
by generalized parabolic theory. The existence of local-in-time classical solutions was proved in
\cite{Bot11}, while the existence of global-in-time weak solutions with general initial
data was shown in \cite{JuSt13}. Concerning the nonisothermal equations, we refer to
\cite{HeJu21}, where an existence analysis for global-in-time weak solutions was presented.
However, this model has some modeling deficiencies explained below. Therefore,
our first aim is to prove the global existence for a thermodynamically consistent 
nonisothermal model.

The uniqueness of strong solutions to the isothermal Maxwell--Stefan equations was
shown in \cite{Bot11,HMPW17,HuSa18}, but the uniqueness of weak solutions for general
coefficients $b_{ij}$ is still unsolved.
A very special case (the coefficients $b_{ij}$ have two degrees of freedom only) 
was investigated in \cite{ChJu18}. It was shown in \cite{HJT22}
that strong solutions are unique in the class of weak solutions, which is known
as the weak--strong uniqueness property. Our second aim is to prove this property for
the nonisothermal case.

Let us detail the main differences of our work compared to \cite{HeJu21}:
\begin{itemize}
\item[(i)] The most important difference is the lack of validity of the Onsager
reciprocity relations in the model of \cite{HeJu21}. The relations imply
the symmetry of the coefficients of the Onsager matrix; see \eqref{1.onsager}. 
The choice in \cite{HeJu21} leads to a cancelation in the entropy inequality, 
thus simplifying the estimation. Our results do not rely on this simplification; 
see Remark \ref{rem.cross} for further details.
\item[(ii)] The constraint \eqref{1.pressure} on the pressure
is not taken into account in \cite{HeJu21}.
This condition is not necessary mathematically, but its lack creates an inconsistency
with the assumption of vanishing barycentric velocity. Indeed, a difference in pressure
induces a force difference, which can result in an acceleration according to Newton's 
second law, if there is no additional force to balance it. 
\item[(iii)] According to Onsager's reciprocity relations, the Onsager matrix $Q$ in 
\eqref{1.onsager} has to be positive semidefinite. We show that $Q$ is in fact positive
definite on the subspace $L=\{\bm{y}\in\R^n:\bm{y}\cdot\sqrt{\bm{\rho}}=0\}$. In \cite{HeJu21},
is is {\em assumed} that this subspace equals $\{\bm{y}\in\R^n:\bm{y}\cdot\bm{1}=0\}$.
This is not consistent with the thermodynamic modeling.
\item[(iv)] We consider different molar masses $m_i$, while they are assumed to be the
same in \cite{HeJu21}. When we assume equal molar masses, the cross-terms cancel, 
and we end up with the simple heat flux $J_e=-\kappa\na\theta$ (see \eqref{1.Je} and
the constraint in \eqref{1.velo}), thus decoupling the equations.
\end{itemize}


\subsection{Main results}

We impose the following assumptions:
\begin{itemize}
\item[(A1)] Domain: $\Omega\subset\R^3$ is a bounded domain with Lipschitz boundary, 
and $T>0$. We set $\Omega_T=\Omega\times(0,T)$ and $\R_+=[0,\infty)$. 
\item[(A2)] Data: $\rho_i^0\in L^\infty(\Omega)$ satisfies $\rho_i^0\ge 0$ in $\Omega$ and
$0 < \rho_* \le \sum_{i=1}^n \rho_i^0\le\rho^*$ in $\Omega$ for some $\rho_*,\rho^*>0$
and for all $i=1,\ldots,n$; $\theta^0\in L^\infty(\Omega)$ satisfies
$\inf_\Omega\theta^0>0$.
\item[(A3)] Coefficients: $b_{ij}=b_{ji}>0$ for all $i,j=1,\ldots,n$.
\item[(A4)] Heat conductivity: $\kappa\in C^0(\R_+^n\times\R_+)$ satisfies
$c_\kappa(1+\theta^2) \le\kappa(\theta)\le C_\kappa(1+\theta^2)$ for some 
$c_\kappa,C_\kappa>0$ and all $(\bm\rho,\theta)\in\R_+^n\times\R_+$.
\end{itemize}

The lower bound for the total mass density $\rho$ is needed to derive uniform estimates for the temperature.
The proof of Lemma 10 in \cite{HJT22} shows that $M_{ij}^{BD}(\bm\rho)$ is bounded
for all $\bm\rho\in\R_+^n$.
The growth condition for the heat conductivity is used to derive higher
integrability bounds for the temperature, which are needed to derive a uniform estimate
for the discrete time derivative of the temperature. We may also assume reaction terms $R_i$
in \eqref{1.mass} with the properties that the total reaction rate $\sum_{i=1}^nR_i$ 
vanishes and the vector of reaction rates $R_i$ is derived from a convex, nonnegative 
potential \cite[Section 2.2]{DDGG20}. 

The first main result is the existence of solutions.

\begin{theorem}[Existence of weak solutions]\label{thm.ex}
Let Assumptions (A1)--(A4) hold. Then there exists a weak solution to
\eqref{1.mass}--\eqref{1.constit} satisfying $\rho_i>0$, $\theta>0$ a.e.\ in
$\Omega_T=\Omega\times(0,T)$ and
\begin{align*}
  & \sqrt{\rho_i}\in L^\infty(\Omega_T)\cap C^0([0,T];L^2(\Omega))\cap 
	L^2(0,T;H^1(\Omega)), \quad
	\pa_t\rho_i\in L^2(0,T;H^1(\Omega)^*), \\
	& \theta\in C^0_w([0,T];L^2(\Omega))\cap L^2(0,T;H^1(\Omega)), \quad
	\pa_t(\rho\theta)\in L^{16/11}(0,T;W^{1,16/11}(\Omega)^*), \\
	& \theta^2,\, \log\theta\in L^2(0,T;H^1(\Omega)), \quad i=1,\ldots,n,
\end{align*}
the weak formulation 
\begin{align*}
  & \int_0^T\langle\pa_t\rho_i,\phi_i\rangle_{H^1(\Omega)^*}\dt 
	+ \int_0^T\int_\Omega\sum_{i,j=}^n
	M_{ij}^{BD}\big(2\na\sqrt{\rho_j}+\rho_j\na\log\theta\big)\dx\dt = 0, \\
	& \int_0^T\int_\Omega\langle\pa_t(\rho\theta),\phi_0\rangle_{W^{1,16/5}(\Omega)^*}\dt 
	+ \int_0^T\int_\Omega\sum_{i,j=1}^n\frac{\theta M_{ij}^{BD}}{m_im_j}\sqrt{\rho_i}
	\big(2\na\sqrt{\rho_j}+\sqrt{\rho_j}\na\log\theta\big)\cdot\na\phi_0\dx\dt \\
	&\phantom{xx}{}+ \int_0^T\int_\Omega\kappa\na\theta\cdot\na\phi_0\dx\dt
	= \lambda\int_0^T\int_{\pa\Omega}(\theta_0-\theta)\phi_0 \ds\dt
\end{align*}
holds for all $\phi_1,\ldots,\phi_n\in L^2(0,T;H^1(\Omega))$ and $\phi_0\in
L^{16/5}(0,T;W^{1,16/5}(\Omega)^*)$, 
and the initial conditions \eqref{1.ic} are satisfied in the sense 
$\rho_i(0)=\rho_i^0$ in $L^2(\Omega)$ and $\theta(0)=\theta^0$ weakly in $L^2(\Omega)$.
\end{theorem}

The idea of the proof is to apply the boundedness-by-entropy method, which automatically
yields $L^\infty(\Omega_T)$ bounds \cite{Jue16}. 
More precisely, we formulate system \eqref{1.mass}--\eqref{1.energy} in terms of
the relative entropy variables $(\mu_i-\mu_n)/\theta$ for $i=1,\ldots,n-1$
and $\log\theta$. We show in Lemma \ref{lem.inv} that this defines the mass densities 
and temperature uniquely as a function of $(w_1,\ldots,w_{n-1},w)$. We introduce
the mathematical entropy density
$$
  h(\bm\rho',\theta) = \sum_{i=1}^{n}\frac{\rho_i}{m_i}\bigg(\log\frac{\rho_i}{m_i}-1\bigg)
	- c_w\rho\log\theta,
$$
where the $n$th partial mass density
is computed from $\rho_n=\rho-\sum_{i=1}^{n-1}\rho_i$, i.e., $h$ depends on
$\bm\rho'=(\rho_1,\ldots,\rho_{n-1})$ and $\theta$. Gradient estimates for
$(\bm\rho,\theta)$ are first derived from the entropy equality
$$
  \frac{\mathrm{d}}{\dt}\int_\Omega h(\bm\rho',\theta)\dx + \int_\Omega\frac{\kappa}{\theta^2}
	|\na\theta|^2\dx + \sum_{i,j=1}^n\int_\Omega M_{ij}^{BD}\frac{d_i}{\theta\sqrt{\rho_i}}
	\frac{d_j}{\theta\sqrt{\rho_j}}\dx = 0,
$$
which becomes an inequality for weak solutions.
Second, as in \cite{HeJu21}, the energy balance equation \eqref{1.energy} yields 
a bound for $\theta^2$ in $L^2(0,T;H^1(\Omega))$.
As mentioned before, the derivation of the entropy inequality differs from that
one in \cite{HeJu21}, because the cross-term
$$
  I_5 = 2\int_\Omega\sum_{i=1}^{n-1}\frac{B_i}{\theta}\na\frac{\mu_i-\mu_n}{\theta}
	\cdot\na\log\theta\dx,
$$
which cancels out in \cite{HeJu21}, 
needs to be controlled. (We recall definition \eqref{1.AB} of $B_i$.)
This is done by observing that the sum $I_4+I_5+I_8$ (see \eqref{3.I458})
is nonnegative,
$$
  I_4+I_5+I_8 = \int_\Omega\sum_{i,j=1}^n A_{ij}
	\na\bigg(\frac{\mu_i}{\theta}+\frac{1}{m_i}\log\theta\bigg)\cdot
	\na\bigg(\frac{\mu_j}{\theta}+\frac{1}{m_j}\log\theta\bigg)\dx \ge 0,
$$
as $(A_{ij})$ is positive semidefinite due to \eqref{3.semidefQ}.

From a technical viewpoint, we approximate equations \eqref{1.mass}--\eqref{1.energy}
by replacing the time derivative by the implicit Euler discretization to avoid issues
with the time regularity and by adding a higher-order regularization 
to achieve $H^2(\Omega)$ and hence $L^\infty(\Omega)$ regularity for the entropy variables.
The approximation is chosen in such a way that a discrete
entropy inequality can be derived, yielding uniform estimates for both
the compactness of the fixed-point operator (to obtain a solution to the approximate
problem) and the de-regularization limit (to obtain a solution to the original problem).

Our second main result concerns the weak--strong uniqueness property.

\begin{theorem}[Weak--strong uniqueness]\label{thm.wsu}
Let the assumptions of Theorem \ref{thm.ex} hold, 
let $\lambda=0$ in \eqref{1.bc}, let $(\bm\rho,\theta)$ be a weak solution
and $(\bar{\bm{\rho}},\bar\theta)$ be
a strong solution to \eqref{1.mass}--\eqref{1.constit}. We assume that there
exist $m,M>0$ such that
$$
  0<\rho_i\le\rho^*,\quad 0<\theta\le M,\quad 0<\bar{\rho}_i\le\rho^*,\quad
	0<m\le\bar\theta\le M\quad\mbox{in }\Omega_T.
$$
Furthermore, we suppose that $\bar{u}_i$, $|\na\log\bar\theta|\in L^\infty(\Omega_T)$
for $i=1,\ldots,n$ and that the thermal conductivity $\kappa$ is Lipschitz continuous.
If the initial data of $(\bm\rho,\theta)$ and $(\bar{\bm{\rho}},\bar\theta)$
coincide then $\bm\rho(x,t)=\bar{\bm{\rho}}(x,t)$ and 
$\theta(x,t)=\bar\theta(x,t)$ for a.e.\ $x\in\Omega$ and all $t>0$.
\end{theorem}

By a strong solution, we understand a solution that has sufficient regularity to
satisfy the entropy equality stated in Lemma \ref{lem.eis}; see Section \ref{sec.wsu}.
Observe that we require the boundedness of the temperature $\theta$, which is not
proved in Theorem \ref{thm.ex}.
The proof of Theorem \ref{thm.wsu} is based on the relative entropy, defined by
\begin{align}
  H(\bm\rho,\theta|\bar{\bm{\rho}},\bar\theta)
	&= \int_\Omega\bigg(h(\bm\rho,\theta) - h(\bar{\bm{\rho}},\bar\theta)
	- \sum_{i=1}^n\frac{\pa h}{\pa\rho_i}(\bar{\bm{\rho}},\bar\theta)(\rho_i-\bar\rho_i)
	- \frac{\pa h}{\pa E}(\bar{\bm{\rho}},\bar\theta)(E-\bar{E})\bigg)\dx \nonumber \\
	&= \int_\Omega\bigg\{\sum_{i=1}^n\frac{1}{m_i}\bigg(\rho_i\log\frac{\rho_i}{\bar\rho_i}
	- (\rho_i-\bar\rho_i)\bigg) - c_w\rho\bigg(\log\frac{\theta}{\bar\theta}
	- (\theta-\bar\theta)\bigg)\bigg\}\dx, \label{1.re}
\end{align}
where $E=c_w\rho\theta$ and $\bar{E}=c_w\rho\bar\theta$ are the internal energy densities. 
The idea is to compute the time derivative:
\begin{align*}
  \frac{\mathrm{d}H}{\dt}(\bm\rho,\theta|\bar{\bm{\rho}},\bar\theta)
	&+ c\int_\Omega\sum_{i=1}^n\rho_i|u_i-\bar{u}_i|^2\dx
	+ c\int_\Omega|\na(\log\theta-\log\bar\theta)|^2\dx \\
	&\le C\int_\Omega\bigg(\sum_{i=1}^n(\rho_i-\bar\rho_i)^2 + (\theta-\bar\theta)^2\bigg)\dx,
\end{align*}
where $c>0$ is some constant and 
$C>0$ depends on the $L^\infty(\Omega_T)$ norms of $\theta$, $\bar{u}_i$, and
$\na\log\bar\theta$, $i=1,\ldots,n$. The difficulty is to estimate the expressions
arising from the time derivative of the relative entropy in such a way that 
only $\bar{u}_i$ and $\bar\theta$ need to be bounded.
Thanks to the positive lower bound for
$\bar\theta$, we can bound the right-hand side in terms of the relative entropy,
$$
  \int_\Omega\bigg(\sum_{i=1}^n(\rho_i-\bar\rho_i)^2 + (\theta-\bar\theta)^2\bigg)\dx 
	\le \int_\Omega H(\bm\rho,\theta|\bar{\bm{\rho}},\bar\theta)\dx.
$$
Then Gronwall's lemma shows that
$H((\bm\rho,\theta)(t)|(\bar{\bm{\rho}},\bar\theta)(t))=0$ for $t>0$ and hence
$(\bm\rho,\theta)(t)=(\bar{\bm{\rho}},\bar\theta)(t)$. 
Compared to \cite{HJT22}, we include the temperature
terms and combine them with the entropy variables $w_i$ in such a way that
the positive semidefiniteness of $M^{BD}$ can be exploited.

The paper is organized as follows. We detail the thermodynamic modeling of equations
\eqref{1.mass}--\eqref{1.constit} in Section \ref{sec.model}. The inversion of the
Maxwell--Stefan system \eqref{1.velo}, the definition of the (relative) entropy variables,
and the formulations of the fluxes in terms of the relative entropy variables, as well
as the corresponding weak formulation is presented in Section \ref{sec.prep}. 
Section \ref{sec.ex} is concerned with the proof of Theorem \ref{thm.ex}, and 
Theorem \ref{thm.wsu} is proved in Section \ref{sec.wsu}.


\section{Modeling}\label{sec.model}

We consider the following system of equations modeling the dynamics of a nonisothermal
gas mixture of $n$ components with mass diffusion and heat conduction:
\begin{align}
  & \pa_t\rho_i + \diver(\rho_i(v+u_i)) = 0, \quad i=1,\ldots,n, \label{2.mass} \\
	& \pa_t(\rho v) + \diver(\rho v\otimes v) = \rho b - \na p, \label{2.mom} \\
	& \pa_t\bigg(\rho e+\frac12\rho|v|^2\bigg) 
	+ \diver\bigg(\bigg(\rho e+\frac12\rho|v|^2\bigg)v\bigg) \label{2.energy} \\
	&\phantom{xx}{}= \diver(\kappa\na\theta) - \diver\sum_{j=1}^n(\rho_ie_i+p_i)u_i
	- \diver(pv) + \rho r + \rho b\cdot v + \sum_{i=1}^n\rho_i b_i\cdot u_i. \nonumber
\end{align}
Besides of the variables introduced in the introduction, $v$ denotes the 
barycentric velocity of the mixture. The quantities $\rho_ib_i$
are the body forces, where $\rho b=\sum_{i=1}^n\rho_ib_i$ is the total force exerted on 
the mixture, and $\rho r$ is the total heat supply due to radiation. 
The diffusional velocities $u_i$, the partial internal energy densities
$\rho_ie_i$, and the partial pressures $p_i$ are determined from the free energy;
see below.

Equations \eqref{2.mass}--\eqref{2.energy} correspond to a so-called class-I model. 
They can be derived either via an entropy invariant model reduction \cite{BoDr15} 
or in the high-friction limit \cite{GeTz23} from a class-II model, 
in which each component has its own velocity $v_i$.
Equations \eqref{2.mass} are the partial mass balances, \eqref{2.mom} is the
momentum balance, and \eqref{2.energy} the energy balance.
As proved in \cite{GeTz23}, system \eqref{2.mass}--\eqref{2.energy} and \eqref{1.velo}
fits into the general theory of hyperbolic--parabolic composite-type systems introduced
in \cite{KaSh88} and further explored in \cite{Ser10}.

As mentioned in the introduction, system \eqref{1.mass}--\eqref{1.energy} 
and \eqref{1.pressure} is supplemented
by the constrained Maxwell--Stefan system \eqref{1.velo} for the velocities $u_i$.
These equations can be derived from a class-II model
in the diffusion approximation \cite[Section 14, (210)]{BoDr15}
or in the high-friction limit \cite[Section 2, (2.50)]{GeTz23} with the driving forces
$$
  d_i = -\frac{\rho_i}{\rho}\na p + \rho_i\theta\na\frac{\mu_i}{\theta}
	- \theta(\rho_ie_i+p_i)\na\frac{1}{\theta} + \rho_i(b-b_i),
$$
where $\mu_i$ is the chemical potential of the $i$th component. Since the pressure is uniform in space, $\na p=0$, 
and we have neglected external forces, the driving force becomes \eqref{1.driving}.
Then equations \eqref{1.mass}--\eqref{1.energy} and \eqref{1.pressure} are obtained
by setting $v=0$ and $r=b_i=0$.

The internal energy densities $\rho_ie_i$, partial pressures $p_i$, and the chemical
potential $\mu_i$ are determined from the Helmholtz free energy. We assume that the
gas is a simple mixture, which implies that these quantities can be calculated from
the partial free energy densities $\psi_i(\rho_i,\theta)$, $i=1,\ldots,n$. We have
$$
  \mu_i = \frac{\pa\psi_i}{\pa\rho_i}, \quad \rho_i\eta_i = -\frac{\pa\psi_i}{\pa\theta}, \quad
	\rho_ie_i = \psi_i + \theta\rho_i\eta_i, \quad p_i = \rho_i\mu_i-\psi_i,
$$ 
where $\rho_i\eta_i$ is the entropy density of the $i$th component and the
equation for $p_i$ is called the Gibbs--Duhem relation.
Defining the partial Helmholtz free energy as 
\begin{equation}\label{2.helm}
  \psi_i = \theta\frac{\rho_i}{m_i}\bigg(\log\frac{\rho_i}{m_i}-1\bigg) 
	- c_w\rho\theta(\log\theta-1), \quad i=1,\ldots,n,
\end{equation}
the thermodynamic quantities are given by \eqref{1.constit}.
Moreover, the driving force $d_i$ and enthalpy $h_i:=\rho_ie_i+p_i$ read as
\begin{equation}\label{2.di}
  d_i = \frac{\na(\rho_i\theta)}{m_i}, \quad h_i = \bigg(c_w+\frac{1}{m_i}\bigg)\rho_i\theta,
	\quad i=1,\ldots,n.
\end{equation}
This corresponds to equations \eqref{1.Je}.


\section{Preparations}\label{sec.prep}
	
\subsection{Inversion of the Maxwell--Stefan system}\label{sec.inv}
	
We discuss the inversion of the Max\-well--Stefan system \eqref{1.velo}
following \cite{GeTz23} and \cite[Section 2]{HJT22}. We write \eqref{1.velo} equivalently as
\begin{equation}\label{3.linsys}
  -\theta\sqrt{\rho_i}\sum_{j=1}^n M_{ij}\sqrt{\rho_j} u_j = d_i, \quad i=1,\ldots,n,
\end{equation}
where the matrix $M(\bm\rho)=(M_{ij})\in\R^{n\times n}$ is given by
\begin{equation}\label{3.M}
  M_{ij} = \begin{cases}
	\sum_{k=1,\,k\neq i}^n b_{ik}\rho_k &\mbox{if }i=j, \\
	-b_{ij}\sqrt{\rho_i\rho_j} &\mbox{if }i\neq j.
	\end{cases}
\end{equation}
We wish to invert $M\bm{v}=\bm{w}$, where $v_i=\sqrt{\rho_i}u_i$ and
$w_i=-d_i/(\theta\sqrt{\rho_i})$. Since $(b_{ij})$ is symmetric,
$0 = (M\bm{v})_i = \sum_{i\neq j}b_{ij}\sqrt{\rho_j}(\sqrt{\rho_j}v_i-\sqrt{\rho_i}v_j)$
shows that the kernel of $M$ consists of $\operatorname{span}\{\sqrt{\bm\rho}\}$.
Thus, we can invert $M$ only on the subspace $L=\{\bm{y}\in\R^n:\sqrt{\bm\rho}\cdot\bm{y}=0\}$.
We define the projections $P_L$ on $L$ and $P_{L^\perp}$ on $L^\perp$ by
$$
  (P_L)_{ij} = \delta_{ij} - \rho^{-1}\sqrt{\rho_i\rho_j}, \quad
	(P_{L^\perp})_{ij} = \rho^{-1}\sqrt{\rho_i\rho_j}\quad\mbox{for }i,j=1,\ldots,n,
$$
where $\delta_{ij}$ is the Kronecker symbol. The matrix $M=(M_{ij})$ is positive
definite on $L$ \cite[Lemma 4]{HJT22}:
\begin{equation}\label{3.posdefM}
  \bm{z}^T M\bm{z} \ge \mu_M|P_L\bm{z}|^2\quad\mbox{for all }\bm{z}\in\R^n,
\end{equation}
where $\mu_M=\min_{i\neq j}b_{ij}>0$. Since the matrix $MP_L+P_{L^\perp}$ is
invertible \cite[Lemma 4]{HJT22}, we can define the Bott--Duffin inverse of $M$ with respect
to $L$ as $M^{BD} = P_L(MP_L+P_{L^\perp})^{-1}$. Hence, we can invert \eqref{3.linsys} by
\begin{equation}\label{3.ui}
  \sqrt{\rho_i}u_i = -\sum_{j=1}^n M_{ij}^{BD}\frac{d_j}{\theta\sqrt{\rho_j}}, \quad
	i=1,\ldots,n.
\end{equation}
The matrix $M^{BD}=M^{BD}(\bm\rho)$ is symmetric and positive definite on $L$ 
\cite[Lemma 4]{HJT22},
\begin{equation}\label{3.semidef}
  \bm{z}^T M^{BD}\bm{z} \ge \mu|P_L \bm{z}|^2 \quad\mbox{for all }\bm{z}\in\R^n,
\end{equation}
where $\mu=(2\sum_{i\neq j}(b_{ij}+1))^{-1}$.


\subsection{Entropy variables}

The mathematical analysis becomes easier when formulating the
system in terms of the so-called entropy variables. To this end, we introduce
the mathematical entropy density
\begin{equation}\label{3.h}
  h = -\sum_{i=1}^n\rho_i\eta_i 
	= \sum_{i=1}^n\frac{\rho_i}{m_i}\bigg(\log\frac{\rho_i}{m_i}-1\bigg)
	- c_w\rho\log\theta,
\end{equation}
which is the negative of the physical (total) entropy density \eqref{1.constit}. 
Summing the mass balances
\eqref{1.mass} over $i=1,\ldots,n$ and using the constraint $\sum_{i=1}^n\rho_iu_i=0$
from \eqref{1.velo}, we obtain $\pa_t\rho=0$. Thus, the total density is determined by the
initial total density, $\rho(x,t)=\sum_{i=1}^n\rho_i^0(x)$ for $x\in\Omega$, 
and is independent of time.
This suggests to compute only the first $n-1$ mass densities, since the last one
can be determined by $\rho_n=\rho-\sum_{i=1}^{n-1}\rho_i$. Then we interpret the
entropy density $h$ as a function of $(\bm\rho',\theta):=(\rho_1,\ldots,\rho_{n-1},\theta)$:
$$
  h(\bm\rho',\theta) = \sum_{i=1}^{n-1}\frac{\rho_i}{m_i}\bigg(\log\frac{\rho_i}{m_i}-1\bigg)
	+ \frac{\rho_n}{m_n}\bigg(\log\frac{\rho_n}{m_n}-1\bigg) - c_w\rho\log\theta
$$
with the partial derivatives
$$
  \frac{\pa h}{\pa\rho_i} = \frac{1}{m_i}\log\frac{\rho_i}{m_i}
	- \frac{1}{m_n}\log\frac{\rho_n}{m_n}, \quad i=1,\ldots,n-1,\quad
	\frac{\pa h}{\pa\theta} = -c_w\frac{\rho}{\theta}.
$$
The Hessian matrix
$$
  \mathrm{D}^2 h = \begin{pmatrix} R & \bm{0} \\ \bm{0}^T & c_w\rho/\theta^2 \end{pmatrix}
	\in\R^{n\times n}, \quad\mbox{where }
	R_{ij} = \frac{\delta_{ij}}{m_i\rho_i} + \frac{1}{m_n\rho_n}, 
$$
is positive definite, showing that the entropy is convex.

According to thermodynamics \cite{BoDr15}, the entropy variables equal
$(\mu_1/\theta,\ldots,\mu_n/\theta,-1/\theta)$. We set
\begin{equation}\label{3.qi}
  q_i = \frac{\mu_i}{\theta} = \frac{1}{m_i}\log\frac{\rho_i}{m_i} - c_w(\log\theta-1)
	\quad\mbox{for }i=1,\ldots,n.
\end{equation}
Since the $n$th partial density is determined by the densities $\rho_1,\ldots,\rho_{n-1}$,
we prefer to work with the relative entropy variables
\begin{equation}\label{3.wi}
  w_i = q_i-q_n = \frac{\mu_i-\mu_n}{\theta} = \frac{\pa h}{\pa\rho_i}, \quad i=1,\ldots,n-1.
\end{equation}
Setting additionally $w=\log\theta$, our new set of variables is $(w_1,\ldots,w_{n-1},w)$.
The following lemma states that the mapping $(\rho_1,\ldots,\rho_n,\theta)\mapsto
(w_1,\ldots,w_{n-1},w)$ is invertible.

\begin{lemma}\label{lem.inv}
Let $(w_1,\ldots,w_{n-1},w)\in\R^n$ and $\rho>0$ be given. Then there there exists a unique 
$(\rho_1,\ldots,\rho_{n},\theta)\in\R_+^{n+1}$ with $\rho_i>0$ for $i=1,\ldots,n$
satisfying $\sum_{i=1}^n\rho_i=\rho$, $w_i=\pa h/\pa\rho_i$ for $i=1,\ldots,n-1$, 
and $w=\log\theta$.
\end{lemma}

\begin{proof}
The proof is similar to \cite[Lemma 6]{ChJu15} with some small changes.
Given $w\in\R$, the temperature equals $\theta=\exp(w)>0$. The function
$$
  f(s) = \sum_{i=1}^{n-1}m_i e^{m_iw_i}\bigg(\frac{\rho-s}{m_n}\bigg)^{m_i/m_n}
	\quad\mbox{for }s\in[0,\rho],
$$
is strictly decreasing and $0=f(\rho)<f(s)<f(0)$ for $s\in(0,\rho)$. 
By continuity, there exists a unique fixed point
$s_0\in(0,\rho)$. Then $\rho_i:=m_i\exp(m_iw_i)((\rho-s_0)/m_n)^{m_i/m_n}$ for
$i=1,\ldots,n$ satisfies $\rho_i>0$ and $\sum_{i=1}^{n-1}\rho_i=f(s_0)=s_0<\rho$.
Consequently, $\rho_n:=\rho-\sum_{i=1}^{n-1}\rho_i=\rho-s_0>0$ and 
$\rho_i/m_i=\exp(m_iw_i)(\rho_n/m_n)^{m_i/m_n}$ is equivalent to
$$
  w_i = \frac{1}{m_i}\log\frac{\rho_i}{m_i}	- \frac{1}{m_n}\log\frac{\rho_n}{m_n}
	= \frac{\pa h}{\pa\rho_i}
$$
for $i=1,\ldots,n-1$, which finishes the proof.
\end{proof}


\subsection{Formulation of the fluxes and parabolicity}\label{sec.flux}

We can compute the fluxes as a linear combination of $\na(w_1,\ldots,w_{n-1},w)$ or
$\na(q_1,\ldots,q_n,-1/\theta)$.

\begin{lemma}
It holds for $i=1,\ldots,n$ that
\begin{align}
  J_i &= -\sum_{j=1}^{n-1}A_{ij}\na w_j - \frac{B_i}{\theta}\na w 
	= -\sum_{j=1}^n A_{ij}\na q_j - B_i\na\bigg({-\frac{1}{\theta}}\bigg), \label{3.Ji} \\
	J_e &= -\kappa\theta\na w - \sum_{j=1}^{n-1}B_j\na w_j
	- \theta\sum_{i,j=1}^n\frac{A_{ij}}{m_im_j}\na w \label{3.Je} \\
	&= -\sum_{j=1}^n B_j\na q_j - \theta^2\bigg(\kappa + \sum_{i,j=1}^n
	\frac{A_{ij}}{m_im_j}\bigg)\na\bigg({-\frac{1}{\theta}}\bigg), \nonumber
\end{align}
where the coefficients
\begin{equation}\label{3.defAB}
  A_{ij} = M_{ij}^{BD}\sqrt{\rho_i\rho_j}, \quad
	B_i = \theta\sum_{j=1}^n A_{ij}\bigg(c_w+\frac{1}{m_j}\bigg) 
	= \theta\sum_{j=1}^n \frac{A_{ij}}{m_j}
\end{equation}
for $i,j=1,\ldots,n$ depend on $(\bm\rho,\theta)$ and satisfy the relations
\begin{equation}\label{3.AB}
  \sum_{i=1}^n A_{ij} = \sum_{j=1}^n A_{ij} = \sum_{i=1}^n B_i = 0.
\end{equation}
\end{lemma}

\begin{proof}
We wish to express the driving force $d_j=\na(\rho_j\theta)/m_j$ from \eqref{1.Je}
in terms of $\na q_j=\na\log\rho_j/m_j-c_w\na\log\theta$. A computation, using
$w=\log\theta$, yields
\begin{equation}\label{3.di}
  d_j = \rho_j\theta\na q_j + \rho_j\theta\bigg(c_w+\frac{1}{m_j}\bigg)\na w.
\end{equation}
Therefore, by \eqref{3.ui}, for $i=1,\ldots,n$,
\begin{align*}
  J_i &= \rho_iu_i = -\sqrt{\rho_i}\sum_{j=}^n M_{ij}^{BD}\frac{d_j}{\theta\sqrt{\rho_j}}
	= -\sum_{j=1}^n M_{ij}^{BD}\sqrt{\rho_i\rho_j}\bigg\{\na q_j + \bigg(c_w+\frac{1}{m_j}\bigg)
	\na w\bigg\} \\
	&= -\sum_{j=1}^n A_{ij}\na q_j - \sum_{j=1}^n A_{ij}\bigg(c_w+\frac{1}{m_j}\bigg)\na\log\theta
	= -\sum_{j=1}^n A_{ij}\na q_j - \frac{B_i}{\theta}\na\log\theta.
\end{align*}
This shows the second relation in \eqref{3.Ji}. The first relation then follows
from \eqref{3.AB} (which is proved below), since, using $q_j=w_j+q_n$ 
for $j=1,\ldots,n-1$ (see \eqref{3.wi}),
\begin{equation}\label{3.relA}
  \sum_{j=1}^n A_{ij}\na q_j = \sum_{j=1}^{n-1}A_{ij}(\na w_j + \na q_n) + A_{in}\na q_n	
	= \sum_{j=1}^{n-1}A_{ij}\na w_j.
\end{equation}

Next, we compute the energy flux defined in \eqref{1.energy}. We use \eqref{2.di},
\eqref{3.ui}, and \eqref{3.di}:
\begin{align*}
  J_e &= -\kappa\theta\na w + \sum_{i=1}^n\sqrt{\rho_i}\theta\bigg(c_w+\frac{1}{m_i}\bigg)
	\sqrt{\rho_i}u_i \\
	&= -\kappa\theta\na w - \theta\sum_{i,j=1}^n\sqrt{\rho_i}\bigg(c_w+\frac{1}{m_i}\bigg)
	M_{ij}^{BD}\frac{d_j}{\theta\sqrt{\rho_j}} \\
	&= -\kappa\theta\na w - \theta\sum_{i,j=1}^n\bigg(c_w+\frac{1}{m_i}\bigg)M_{ij}^{BD}
	\sqrt{\rho_i\rho_j}\bigg\{\na q_j + \bigg(c_w+\frac{1}{m_j}\bigg)\na w\bigg\} \\
	&= -\kappa\theta\na w - \sum_{j=1}^n B_j\na q_j - \theta\sum_{i,j=1}^n A_{ij}
	\bigg(c_w+\frac{1}{m_i}\bigg)\bigg(c_w+\frac{1}{m_j}\bigg)\na w \\
	&= -\kappa\theta\na w - \sum_{j=1}^n B_j\na q_j - \theta\sum_{i,j=1}^n 
	\frac{A_{ij}}{m_im_j}\na w,
\end{align*}
where the last equation follows from \eqref{3.AB}. Moreover, because of
\begin{equation}\label{3.relB}
  \sum_{j=1}^n B_j\na q_j = \sum_{j=1}^{n-1}B_j\na(w_j+q_n) + B_n\na q_n
	= \sum_{j=1}^{n-1}B_j\na w_j,
\end{equation}
we have proved \eqref{3.Je}. 

It remains to verify \eqref{3.AB}. We recall the property
$P_L(MP_L+P_{L^\perp})^{-1}P_{L^\perp}=0$ from \cite[Lemma 2]{Yon90}, which implies
that $M^{BD}P_{L^\perp}=0$. Hence, $L^\perp\subset\operatorname{ker}M^{BD}$ and
since $L^\perp=\operatorname{span}\{\sqrt{\bm\rho}\}$, we conclude that
$\sum_{j=1}^n M_{ij}^{BD}\sqrt{\rho_j}=0$. This shows that, by the definition of $A_{ij}$,
$$
  \sum_{j=1}^n A_{ij} = \sqrt{\rho_i}\sum_{j=1}^n M_{ij}^{BD}\sqrt{\rho_j} = 0.
$$
The symmetry of $(A_{ij})$ immediately gives $\sum_{i=1}^n A_{ij}=0$. Finally, by the
definition of $B_i$,
\begin{align*}
  \sum_{i=1}^n B_i = \theta\sum_{i,j=1}^n M_{ij}^{BD}\sqrt{\rho_i\rho_j}
	\bigg(c_w+\frac{1}{m_j}\bigg)
	= \theta\sum_{j=1}^n\bigg(c_w+\frac{1}{m_j}\bigg)\sum_{i=1}^n A_{ij}=0.
\end{align*}
This finishes the proof.
\end{proof}

The previous proof shows that we can formulate the diffusion fluxes in different ways.

\begin{corollary}\label{coro.flux}
It holds for $i=1,\ldots,n$ that
$$
  J_i = \rho_iu_i = -\sum_{j=1}^n A_{ij}\na\bigg(q_j+\frac{w}{m_j}\bigg)
	= -\sqrt{\rho_i}\sum_{j=1}^n M_{ij}^{BD}\frac{d_j}{\theta\sqrt{\rho_j}}.
$$
\end{corollary}

We claim that the Onsager matrix $Q\in\R^{(n+1)\times(n+1)}$ in \eqref{1.onsager}
is positive semidefinite. Let $a = \theta(\kappa + \sum_{i,j=1}^n A_{ij}/(m_im_j))$.
We compute for $\xi\in\R^{n+1}$:
\begin{align}\label{3.semidefQ}
  \xi^T Q\xi &= \sum_{i,j=1}^n A_{ij}\xi_i\xi_j
	+ 2\sum_{i=1}^n B_i\xi_i\xi_{n+1} + a\xi_{n+1}^n \\
	&= \sum_{i,j=1}^n A_{ij}\xi_i\xi_j 
	+ 2\theta\sum_{i,j=1}^n \frac{A_{ij}}{m_j}\xi_i\xi_{n+1}
	+ \theta^2\bigg(\kappa+\sum_{i,j=1}^n \frac{A_{ij}}{m_im_j}\bigg)\xi_{n+1}^2 \nonumber \\
	&= \sum_{i,j=1}^n A_{ij}\bigg(\xi_i + \frac{\theta\xi_{n+1}}{m_i}\bigg)
	\bigg(\xi_j + \frac{\theta\xi_{n+1}}{m_j}\bigg) + \kappa\theta^2\xi_{n+1}^2 \ge 0,
	\nonumber
\end{align}
where the nonnegativity follows from the positive semidefiniteness \eqref{3.semidef}
of $M^{BD}$. This reveals the parabolicity of our system in terms
of the entropy variables.


\subsection{Weak formulation}

The previous subsection shows that we can write our system as the mass and energy balances
\eqref{1.mass}--\eqref{1.energy} with the fluxes \eqref{3.Ji}--\eqref{3.Je}.
The weak formulation in the relative entropy variables \eqref{3.wi} reads as
\begin{align}
  & \int_0^T\langle\pa_t\rho_i,\phi_i\rangle\dt
	+ \int_0^T\int_\Omega\bigg(\sum_{j=1}^{n-1}A_{ij}\na w_j + e^{-w}B_i\na w\bigg)
	\cdot\na\phi_i \dx\dt = 0, \label{3.mass} \\
	& \int_0^T\langle\pa_t E,\phi_0\rangle\dt + \int_0^T\int_\Omega e^w\bigg(
	\kappa + \sum_{i,j=1}^n\frac{A_{ij}}{m_im_j}\bigg)\na w\cdot\na\phi_0\dx\dt 
	\label{3.energy} \\
	&\phantom{xx}{}+ \int_0^T\int_\Omega\sum_{j=1}^{n-1}B_j\na w_j\cdot\na\phi_0 \dx\dt
	= \lambda\int_0^T\int_{\pa\Omega}(\theta_0-\theta)\phi_0 \ds\dt \nonumber 
\end{align}
for test functions $\phi_1,\ldots,\phi_n\in L^2(0,T;H^1(\Omega))$ and
$\phi_0\in L^\infty(0,T;W^{1,\infty}(\Omega))$.
According to \eqref{1.constit}, the energy is given by 
$E=c_w\rho\theta$. Moreover, $\rho_i$, $A_{ij}$, $B_i$, and $E$ are interpreted
as functions of $(w_1,\ldots,w_{n-1},w)$.


\section{Proof of Theorem \ref{thm.ex}}\label{sec.ex}

The proof follows the lines of \cite[Section 3]{HeJu21}, which is based on the
boundedness-by-entropy method \cite{Jue16}, but some details are different.
We approximate equations \eqref{3.mass}--\eqref{3.energy} by replacing the time derivative
by the implicit Euler scheme and adding a higher-order regularization in $w_i$. The
existence of solutions to the approximate system is shown by means of the Leray--Schauder
fixed-point theorem, where the compactness of the fixed-point operator is obtained
by the approximate entropy inequality. This inequality yields estimates uniform in the
regularization parameters, allowing for the de-regularization limit via the
Aubin--Lions compactness lemma. 

Let $\eps\in(0,1)$, $N\in\N$, and $\tau=T/N$. We set $w_0=\log\theta_0$ and
$\bm{w}=(w_1,\ldots,w_{n-1},w)$. Let $\bar{\bm{w}}=(\bar{w}_1,\ldots,\bar{w}_{n-1},\bar{w})
\in L^\infty(\Omega;\R^{n})$ be given. We define for test functions $\phi_i\in H^2(\Omega)$, 
$i=0,\ldots,n-1$, the approximate scheme
\begin{align}\label{3.approx1}
  0 &= \frac{1}{\tau}\int_\Omega\big(\rho_i(\bm{w})-\rho_i(\bar{\bm{w}})\big)\phi_i\dx
	+ \int_\Omega\bigg(\sum_{j=1}^{n-1}A_{ij}\na w_j + e^{-w}B_i\na w\bigg)
	\cdot\na\phi_i\dx \\
	&\phantom{xx}{}+ \eps\int_\Omega\big(\mathrm{D}^2w_i:\mathrm{D}^2\phi_i + w_i\phi_i\big)\dx, 
	\nonumber \\
	0 &= \frac{1}{\tau}\int_\Omega\big(E(\bm{w})-E(\bar{\bm{w}})\big)\phi_0\dx
	+ \int_\Omega e^w\bigg(\kappa(e^w) + \sum_{i,j=1}^n\frac{A_{ij}}{m_im_j}
	\bigg)\na w\cdot\na\phi_0\dx \label{3.approx2} \\
	&\phantom{xx}{}+ \int_\Omega\sum_{i=1}^{n-1}B_i\na w_i\cdot\na\phi_0\dx
	- \lambda\int_{\pa\Omega}(e^{w_0}-e^w)\phi_0\ds \nonumber \\
	&\phantom{xx}{}+ \eps\int_\Omega(e^{w_0}+e^w)(w-w_0)\phi_0\dx
	+ \eps\int_\Omega e^w\big(\mathrm{D}^2 w:\mathrm{D}^2\phi_0 + |\na w|^2\na w\cdot\na\phi_0
	\big)\dx, \nonumber 
\end{align}
where $\mathrm{D}^2 w_i$ is the Hesse matrix of $w_i$, the double point ``:'' denotes
the Frobenius matrix product, we recall that $E(\bm{w})=c_w\rho\theta$, and
$A_{ij}$ and $B_i$ are interpreted as functions of $\bm{w}$.
The higher-order regularization yields solutions $w_i,w\in H^2(\Omega)$, and
the $W^{1,4}(\Omega)$ regularization allows us to estimate the higher-order terms
when using the test function $e^{-w_0}-e^{-w}$ (see the estimate of $I_{11}$ below).
The lower-order regularization $(e^{w_0}-e^w)(w-w_0)$ provides an $\eps$-dependent
$L^2(\Omega)$ bound for $w$. 

\subsection{Solution of the linearized approximate problem}

Let $\bm{w}^*\in W^{1,4}(\Omega;$ $\R^n)$ and $\sigma\in[0,1]$. We want to find
a solution $\bm{w}\in H^2(\Omega;\R^n)$ to the linear problem
\begin{equation}\label{3.LM}
  a(\bm{w},\bm\phi) = \sigma F(\bm\phi)\quad\mbox{for }
	\bm\phi=(\phi_1,\ldots,\phi_{n-1},\phi_0)\in H^2(\Omega;\R^n),
\end{equation}
where
\begin{align*}
  a(\bm{w},\bm\phi) &= \int_\Omega\kappa(e^{w^*})e^{w^*}\na w\cdot\na\phi_0\dx
	+ \eps\int_\Omega\sum_{i=1}^{n-1}\big(\mathrm{D}^2 w_i:\mathrm{D}^2\phi_i 
	+ w_i\phi_i\big)\dx \\
	&\phantom{xx}{}+ \eps\int_\Omega(e^{w_0}+e^{w^*})w\phi_0 \dx
	+ \eps\int_\Omega e^{w^*}\big(\mathrm{D}^2 w:\mathrm{D}^2\phi_0
	+ |\na w^*|^2\na w\cdot\na\phi_0\big)\dx, \\
	F(\bm\phi) &= -\int_\Omega\sum_{i,j=1}^{n-1}A_{ij}(\bm{w}^*)\na w_j^*\cdot\na\phi_i\dx
	- \int_\Omega e^{w^*}\sum_{i,j=1}^n\frac{A_{ij}(\bm{w}^*)}{m_im_j}\na w^*\cdot\na\phi_0\dx \\
	&\phantom{xx}{}- \int_\Omega\sum_{i=1}^{n-1}B_i(\bm{w}^*)e^{-w^*}\na w^*
	\cdot\na\phi_i\dx - \int_\Omega\sum_{i=1}^{n-1}B_i(\bm{w}^*)\na w_i^*\cdot\na\phi_0\dx \\
	&\phantom{xx}{}- \frac{1}{\tau}\int_\Omega\sum_{i=1}^{n-1}(\rho_i^*-\bar{\rho}_i)\phi_i\dx
	- \frac{1}{\tau}\int_\Omega(E^*-\bar{E})\phi_0\dx
	+ \lambda\int_{\pa\Omega}(e^{w_0}-e^{w^*})\phi_0 \ds \\
	&\phantom{xx}{}+ \eps\int_\Omega(e^{w_0}+e^{w^*})w_0\phi_0\dx,
\end{align*}
where we abbreviated $\rho_i^*=\rho_i(\bm{w}^*)$, $\bar{\rho}_i=\rho_i(\bar{\bm{w}})$,
$E^*=c_w\rho e^{w^*}$, and $\bar{E}=c_w\rho e^{\bar{w}}$.
The bilinear form $a$ is clearly coercive on $H^2(\Omega;\R^n)$, and both $a$ and $F$
are continuous on this space. By the Lax--Milgram lemma,
there exists a unique solution $\bm{w}\in H^2(\Omega;\R^n)$ to \eqref{3.LM}.


\subsection{Solution of the approximate problem}\label{sec.LM}

The solution $\bm{w}\in H^2(\Omega;\R^n)$ to \eqref{3.LM} 
defines the fixed-point operator $S:W^{1,4}(\Omega;\R^n)
\times[0,1]\to W^{1,4}(\Omega;\R^n)$, $S(\bm{w}^*,\sigma)=\bm{w}$. 
The operator is continuous, compact (because of the compact embedding 
$H^2(\Omega;\R^n)\hookrightarrow W^{1,4}(\Omega;\R^n)$), and it
satisfies $S(\bm{w}^*,0)=0$ for all $\bm{w}^*\in W^{1,4}(\Omega;\R^n)$. It remains to
find a uniform bound for all fixed points of $S(\cdot,\sigma)$. Let $\bm{w}
\in H^2(\Omega;\R^n)$ be such a fixed point. Then $\bm{w}$ solves \eqref{3.LM} with
$\bm{w}^*=\bm{w}$. We choose the test functions $\phi_i=w_i$ for $i=1,\ldots,n-1$ and
$\phi_0=e^{-w_0}-e^{-w}$ in \eqref{3.LM}:
\begin{align}\label{3.est}
  0 &= \frac{\sigma}{\tau}\int_\Omega\sum_{i=1}^{n-1}(\rho_i-\bar\rho_i)w_i\dx
	+ \frac{\sigma}{\tau}\int_\Omega(E-\bar{E})(-e^{-w})\dx
	+ \frac{\sigma}{\tau}\int_\Omega(E-\bar{E})e^{-w_0}\dx \\
	&\phantom{xx}{}+ \sigma\int_\Omega\sum_{i,j=1}^{n-1}A_{ij}(\bm{w})\na w_i\cdot\na w_j\dx
	+ 2\sigma\int_\Omega\sum_{i=1}^{n-1}B_i(\bm{w})e^{-w}\na w_i\cdot\na w\dx \nonumber \\
	&\phantom{xx}{}+ \int_\Omega\kappa(e^w)|\na w|^2\dx 
	+ \eps\int_\Omega\sum_{i=1}^{n-1}\big(|\mathrm{D}^2 w_i|^2 + w_i^2\big)\dx
	+ \sigma\int_\Omega\sum_{i,j=1}^n \frac{A_{ij}(\bm{w})}{m_im_j}|\na w|^2\dx \nonumber \\
	&\phantom{xx}{}- \sigma\lambda\int_{\pa\Omega}(e^{w_0}-e^w)(e^{-w_0}-e^{-w})\ds
	+ \eps\int_\Omega(e^{w_0}+e^w)(e^{-w_0}-e^{-w})(w-\sigma w_0)\dx \nonumber \\
	&\phantom{xx}{}+ \eps\int_\Omega\big(|\mathrm{D}^2 w|^2 - \mathrm{D}w:(\na w\otimes\na w)
	+ |\na w|^4\big)\dx =: I_1+\cdots+I_{11}. \nonumber 
\end{align}
We estimate the terms $I_1,\ldots,I_{11}$ step by step. First, by the convexity of
the entropy and arguing similarly as in \cite[Section 3, Step 2]{HeJu21},
\begin{align*}
  I_1 + I_2 &= \frac{\sigma}{\tau}\int_\Omega\sum_{i=1}^{n-1}\bigg(
	(\rho_i-\bar{\rho}_i)\frac{\pa h}{\pa\rho_i} + (\theta-\bar\theta)\frac{\pa h}{\pa\theta}
	\bigg)\dx \\
	&\ge \frac{\sigma}{\tau}\int_\Omega\big(h(\rho_1,\ldots,\rho_{n-1},\theta)
	- h(\bar\rho_1,\ldots,\bar\rho_{n-1},\bar\theta)\big)\dx,
\end{align*}
where we have set $\theta=e^w$ and $\bar\theta=e^{\bar{w}}$.
Definition \eqref{3.wi} of $w_i$, definition \eqref{3.defAB} of $B_i$, and the relations
$$
  \sum_{j=1}^{n-1}A_{ij}(\bm{w})\na w_j = \sum_{j=1}^n A_{ij}(\bm{w})\na q_j, \quad 
	\sum_{i=1}^{n-1}B_{i}(\bm{w})\na w_i = \sum_{j=1}^n B_{i}(\bm{w})\na q_i
$$
from \eqref{3.relA}--\eqref{3.relB} allow us to rewrite
the sum $I_4+I_5+I_8$ as
\begin{equation}\label{3.I458}
  I_4+I_5+I_8 = \sigma\int_\Omega\sum_{i,j=1}^n A_{ij}(\bm{w})
	\na\bigg(q_i + \frac{w}{m_i}\bigg)\cdot\na\bigg(q_j + \frac{w}{m_j}\bigg)\dx.
\end{equation}
This expression is nonnegative because of the positive semidefiniteness of 
$A_{ij}=M_{ij}^{BD}\sqrt{\rho_i\rho_j}$; see \eqref{3.semidef}. Furthermore,
since $\sinh(z)/z\ge 1$ for $z\in\R$, $z\neq 0$,
\begin{align*}
  I_9 &= \sigma\lambda\int_{\pa\Omega} e^{-w-w_0}(e^w-e^{w_0})^2\dx \ge 0, \\
  I_{10} &= 2\eps\int_\Omega\sinh(w-w_0)(w-\sigma w_0)\dx
	= 2\eps\int_\Omega(w-w_0)(w-\sigma w_0)\frac{\sinh(w-w_0)}{w-w_0}\dx \\
	&= \eps\int_\Omega w^2\frac{\sinh(w-w_0)}{w-w_0}\dx 
	+ \eps\int_\Omega\big(w^2 - 2(1+\sigma)ww_0 + 2\sigma w_0^2\big)
	\frac{\sinh(w-w_0)}{w-w_0}\dx \\
	&\ge \eps\int_\Omega w^2\dx + \eps\int_\Omega\big(w^2 - 2(1+\sigma)ww_0 + 2\sigma w_0^2\big)
	\frac{\sinh(w-w_0)}{w-w_0}\dx.
\end{align*}
We claim that there exists $m=m(w_0,\sigma)>0$ such that for all $w\in\R$,
$$
  g(w) = \big(w^2 - 2(1+\sigma)ww_0 + 2\sigma w_0^2\big)
	\frac{\sinh(w-w_0)}{w-w_0} \ge -m,
$$
where $w_0\in\R$ and $\sigma\in(0,1]$ are given. 
Indeed, this follows from $g(w)\to\infty$ as $|w|\to\infty$ and $g((1+\sigma)w_0)<0$
(unless $w_0=0$). We conclude that
$$
  I_{10} \ge \eps\int_\Omega w^2\dx - \eps m.
$$
Finally, we can estimate
$$
  I_{11} = \frac{\eps}{2}\int_\Omega\big(|\mathrm{D}^2 w|^2 
	+ |\mathrm{D}^2 w-\na w\otimes\na w|^2 + |\na w|^4\big)\dx
	\ge \frac{\eps}{2}\int_\Omega\big(|\mathrm{D}^2 w|^2 + |\na w|^4\big)\dx.
$$
Summarizing these estimates, we find that
\begin{align}\label{3.est2}
  \frac{\sigma}{\tau}&\int_\Omega\big(h(\rho_1,\ldots,\rho_{n-1},\theta)+Ee^{-w_0}\big)\dx
	+ \eps C\big(\|\bm{w}\|_{H^2(\Omega)}^2 + \|\na w\|_{L^{4}(\Omega)}^4\big) \\
	&{}+ \int_\Omega\kappa(e^w)|\na w|^2\dx
	\le \frac{\sigma}{\tau}\int_\Omega\big(h(\bar\rho_1,\ldots,\bar\rho_{n-1},\bar\theta)
	+\bar{E}e^{-w_0}\big)\dx + \eps m. \nonumber
\end{align}
The right-hand side is bounded since $\bar{\bm{w}}\in
L^\infty(\Omega;\R^n)$ by assumption, implying that $(\bar\rho_1,\ldots,\bar\rho_{n-1},
\bar\theta)\in L^\infty(\Omega;\R^n)$. The first term on the left-hand side is bounded from
below since, by definition \eqref{3.h} of $h$ and $E e^{-w_0}=c_w\rho\theta/\theta_0$,
$$
  h(\rho_1,\ldots,\rho_{n-1},\theta)+Ee^{-w_0}
	= \sum_{i=1}^{n}\frac{\rho_i}{m_i}\bigg(\log\frac{\rho_i}{m_i}-1\bigg)
	- c_w\rho\bigg(\log\theta-\frac{\theta}{\theta_0}\bigg). 
$$
Thus, we obtain a uniform bound for $\bm{w}$ in $H^2(\Omega;\R^n)$ and
consequently also in $W^{1,4}(\Omega;\R^n)$. We can apply the Leray--Schauder 
fixed-point theorem to conclude the existence of a fixed point of $S(\cdot,1)$.
This, in turn, shows that $\bm{w}$ is a weak solution to the approximate problem 
\eqref{3.approx1}--\eqref{3.approx2}. 

\begin{remark}[Treatment of the cross-terms]\label{rem.cross}\rm
In the paper \cite{HeJu21}, the fluxes are given by
$$
  \begin{pmatrix} \bm{J} \\ J_e \end{pmatrix}
	= -\begin{pmatrix} M & -\bm{G} \\ \bm{G}^T & \kappa\theta^2 \end{pmatrix}
	\na\begin{pmatrix} \bm\mu/\theta \\ -1/\theta \end{pmatrix},
$$
where $M=M(\bm\rho,\theta)\in\R^{n\times n}$ and $\bm{G}=\bm{G}(\bm\rho,\theta)\in\R^n$.
A multiplication of this equation by $\na(\bm\mu/\theta,-1/\theta)$ 
shows that the cross-terms cancel out,
$$
  -\na\begin{pmatrix} \bm\mu/\theta \\ -1/\theta \end{pmatrix}^T
  :\begin{pmatrix} \bm{J} \\ J_e \end{pmatrix}
	= \sum_{i,j=1}^n M_{ij}\na\frac{\mu_i}{\theta}\cdot\na\frac{\mu_j}{\theta}
	+ \kappa|\na\log\theta|^2 \ge 0,
$$
since $M$ is assumed to be positive semidefinite in \cite{HeJu21}. In the present work, we have
$$
  \begin{pmatrix} \bm{J} \\ J_e \end{pmatrix}
	= -\begin{pmatrix} A & \bm{B} \\ \bm{B}^T & a
	\end{pmatrix}\na\begin{pmatrix} \bm\mu/\theta \\ -1/\theta \end{pmatrix},
$$
and the cross-terms do not cancel. This is compensated by the sum 
$\sum_{i,j=1}^n A_{ij}/(m_im_j)$. Indeed, a computation shows that (also see \eqref{3.I458})
\begin{align*}
  -\na\begin{pmatrix} \bm\mu/\theta \\ -1/\theta \end{pmatrix}^T
  :\begin{pmatrix} \bm{J} \\ J_e \end{pmatrix}
	&= \sum_{i,j=1}^n A_{ij}\na\bigg(q_i + \frac{w}{m_i}\bigg)
	\cdot\na\bigg(q_j + \frac{w}{m_j}\bigg) 
	+ \kappa|\na\log\theta|^2 \ge 0,
\end{align*}
since $A$ is positive semidefinite because of \eqref{3.semidefQ}.
\qed\end{remark}


\subsection{Discrete entropy inequality}

We derive some estimates from \eqref{3.est} with $\sigma=1$, which are
uniform in $(\eps,\tau)$,
by exploiting the sum $I_4+I_5+I_8$, which we have neglected in \eqref{3.est2}.
Taking into account that the estimate of $I_{10}$ becomes for $\sigma=1$
$$
  I_{10} = 2\eps\int_\Omega\sinh(w-w_0)(w-w_0)\dx 
	\ge 2\eps\int_\Omega(w-w_0)^2\dx\ge 0,
$$
we obtain the discrete entropy inequality
\begin{align}\label{3.dei}
  \frac{\sigma}{\tau}&\int_\Omega\big(h(\rho_1,\ldots,\rho_{n-1},\theta)+Ee^{-w_0}\big)\dx
	+ \eps C\big(\|\bm{w}\|_{H^2(\Omega)}^2 + \|\na w\|_{L^{4}(\Omega)}^4\big) \\
	&\phantom{xx}{}{}+ \int_\Omega\kappa(e^w)|\na w|^2\dx
	+ \int_\Omega\sum_{i,j=1}^n A_{ij}\na\bigg(q_i+\frac{w}{m_i}\bigg)\cdot\na\bigg(
	q_j+\frac{w}{m_j}\bigg)\dx \nonumber \\
	&\le \frac{\sigma}{\tau}\int_\Omega\big(h(\bar\rho_1,\ldots,\bar\rho_{n-1},\bar\theta)
	+\bar{E}e^{-w_0}\big)\dx. \nonumber
\end{align}

\begin{lemma}\label{lem.posdef}
It holds that
\begin{equation}\label{3.posdef}
  \int_\Omega\sum_{i,j=1}^n A_{ij}\na\bigg(q_i+\frac{w}{m_i}\bigg)\cdot\na\bigg(
	q_j+\frac{w}{m_j}\bigg)\dx \ge \int_\Omega\sum_{i=1}^n\frac{\mu}{m_i^2}
	|2\na\sqrt{\rho_i}+\sqrt{\rho_i}\na w|^2\dx,
\end{equation}
where $\mu>0$ is defined in \eqref{3.semidef}.
\end{lemma}

We deduce from Assumption (A4) that $\kappa(e^w)|\na w|^2\ge c_\kappa |\na w|^2$,
and in view of \eqref{3.dei}, this quantity is bounded in $L^2(\Omega)$.
Therefore, Lemma \ref{lem.posdef} yields a gradient bound for $\sqrt{\rho_i}$
in $L^2(\Omega)$, since
$$
  4|\na\sqrt{\rho_i}|^2 \le |2\na\sqrt{\rho_i}+\sqrt{\rho_i}\na w|^2
	+ \rho_i|\na w|^2.
$$

\begin{proof}[Proof of Lemma \ref{lem.posdef}]
It follows from \eqref{3.qi} and \eqref{3.AB} that
$$
  \sum_{i,j=1}^n A_{ij}\na q_i = \sum_{i,j=1}^n A_{ij}\frac{\na\log\rho_i}{m_i}
	- c_w\sum_{i,j=1}^n A_{ij}\na w = \sum_{i,j=1}^n A_{ij}\frac{\na\rho_i}{m_i\rho_i}
$$
and therefore, in view of the definition $A_{ij}=M_{ij}^{BD}\sqrt{\rho_i\rho_j}$ and
the positive definiteness \eqref{3.semidef} on the subspace $L$,
\begin{align*}
  \sum_{i,j=1}^n& A_{ij}\na\bigg(q_i+\frac{w}{m_i}\bigg)\cdot\na\bigg(
	q_j+\frac{w}{m_j}\bigg) 
	= \sum_{i,j=1}^n A_{ij}\bigg(\frac{\na\rho_i}{m_i\rho_i}+\frac{\na w}{m_i}\bigg)
	\cdot\bigg(\frac{\na\rho_j}{m_j\rho_j}+\frac{\na w}{m_j}\bigg) \\
	&= \sum_{i,j=1}^n M_{ij}^{BD}
	\frac{1}{m_i}\bigg(\frac{\na\rho_i}{\sqrt{\rho_i}}+\sqrt{\rho_i}\na w\bigg)\cdot
	\frac{1}{m_j}\bigg(\frac{\na\rho_j}{\sqrt{\rho_j}}+\sqrt{\rho_j}\na w\bigg) \\
	&\ge \mu\bigg|P_L\bigg(\frac{1}{m_i}\bigg(\frac{\na\rho_i}{\sqrt{\rho_i}}+\sqrt{\rho_i}\na w
	\bigg)\bigg)_{i=1}^n\bigg|^2.
\end{align*}
We insert the definition of the projection matrix $P_L$:
\begin{align*}
  \bigg[P_L&\bigg(\frac{1}{m_j}\bigg(\frac{\na\rho_j}{\sqrt{\rho_j}}+\sqrt{\rho_j}\na w
	\bigg)\bigg)_{j=1}^n\bigg]_i 
	= \sum_{j=1}^n\bigg(\delta_{ij} - \frac{\sqrt{\rho_i\rho_j}}{\rho}
	\bigg)\frac{1}{m_j}\bigg(\frac{\na\rho_j}{\sqrt{\rho_j}}+\sqrt{\rho_j}\na w\bigg) \\
	&= \frac{1}{m_i}\bigg(\frac{\na\rho_i}{\sqrt{\rho_i}}+\sqrt{\rho_i}\na w\bigg)
	- \frac{\sqrt{\rho_i}}{\rho}\sum_{j=1}^n\frac{1}{m_j}(\na\rho_j+\rho_j\na w)
	= \frac{1}{m_i}\bigg(\frac{\na\rho_i}{\sqrt{\rho_i}}+\sqrt{\rho_i}\na w\bigg).
\end{align*}
The last step follows from the pressure constraint \eqref{1.pressure}.
Indeed, by \eqref{1.constit}, 
\begin{equation}\label{3.zero}
  \sum_{j=1}^n\frac{1}{m_j}(\na\rho_j+\rho_j\na w)
	= \frac{1}{\theta}\sum_{j=1}^n \frac{\na(\rho_j\theta)}{m_j} = \frac{1}{\theta}\na p = 0.
\end{equation}
We have shown that
$$
  \sum_{i,j=1}^n A_{ij}\na\bigg(q_i+\frac{w}{m_i}\bigg)\cdot\na\bigg(
	q_j+\frac{w}{m_j}\bigg) \ge \sum_{i=1}^n\frac{\mu}{m_i^2}
	\big|2\na\sqrt{\rho_i}+\sqrt{\rho_i}\na w\big|^2,
$$
which equals \eqref{3.posdef} after integration over $\Omega$.
\end{proof}

\begin{remark}\rm
We observe that the sum \eqref{3.zero} vanishes even without requiring the
constraint \eqref{1.pressure}. Indeed, by \eqref{2.di},
$$
  \sum_{j=1}^n\frac{1}{m_j}(\na\rho_j+\rho_j\na w)
	= \frac{1}{\theta}\sum_{j=1}^n\frac{1}{m_j}\na(\rho_j\theta)
	= \frac{1}{\theta}\sum_{j=1}^n d_j = 0.
$$
The fact that $\sum_{j=1}^n d_j$ vanishes is a necessary condition for the invertibility
of the linear system \eqref{3.linsys}.
\qed\end{remark}

In view of Lemma \ref{lem.posdef} and the lower bound $\kappa\ge c_\kappa(1+\theta^2)$,
we conclude from \eqref{3.dei} the following discrete entropy inequality.

\begin{lemma}[Discrete entropy inequality]\label{lem.dei}
It holds that
\begin{align*}
  \frac{1}{\tau}&\int_\Omega\big(h(\rho_1,\ldots,\rho_{n-1},\theta)+Ee^{-w_0}\big)\dx
	+ \eps C\big(\|\bm{w}\|_{H^2(\Omega)}^2 + \|\na w\|_{L^{4}(\Omega)}^4\big) \\
	&\phantom{xx}{}{}+ \int_\Omega\big(|\na w|^2 + |\na\theta|^2\big)\dx
	+ \int_\Omega\sum_{i=1}^n\frac{\mu}{m_i^2}
	\big|2\na\sqrt{\rho_i} + \sqrt{\rho_i}\na w\big|^2\dx \\
	&\le \frac{1}{\tau}\int_\Omega\big(h(\bar\rho_1,\ldots,\bar\rho_{n-1},\bar\theta)
	+\bar{E}e^{-w_0}\big)\dx.
\end{align*}
\end{lemma}

Finally, we derive an estimate for the temperature.

\begin{lemma}\label{lem.temp}
There exists a constant $C>0$, only depending on $\lambda$, $\Omega$, $\pa\Omega$, and
$\theta^0$ such that
$$
  \frac{c_w}{2\tau}\int_\Omega\rho\theta^2\dx + \frac{c_\kappa}{2}\int_\Omega
	(1+\theta^2)|\na\theta|^2\dx \le C + C\int_\Omega\sum_{i=1}^n|\na\sqrt{\rho_i}|^2\dx
	+ \frac{c_w}{2\tau}\int_\Omega\rho\bar{\theta}^2\dx.
$$
\end{lemma}

\begin{proof}
We use $\theta$ as a test function in the approximate energy equation \eqref{3.approx2}.
Observing that $\na w_i=\na\rho_i/(m_i\rho_i)-\na\rho_n/(m_n\rho_n)$ by \eqref{3.wi} and
$\sum_{i=1}^n B_i\na w_i=\sum_{i=1}^n B_i(m_i\rho_i)^{-1}\na\rho_i$ by \eqref{3.AB},
we find that 
\begin{align*}
  0 &= \frac{c_w}{\tau}\int_\Omega\rho(\theta-\bar\theta)\dx
	+ \int_\Omega\kappa(\theta)|\na\theta|^2\dx
	+ \int_\Omega\sum_{i,j=1}^n\frac{A_{ij}}{m_im_j}|\na\theta|^2\dx \\
	&\phantom{xx}{}\int_\Omega\sum_{i=1}^n \frac{B_i}{m_i\rho_i}\na\rho_i\cdot\na\theta\dx 
	- \lambda\int_{\pa\Omega}(\theta_0-\theta)\theta\ds
	+ \eps\int_\Omega(\theta_0+\theta)(\log\theta-\log\theta_0)\theta\dx \\
	&\phantom{xx}{}+ \eps\int_\Omega\bigg(|\mathrm{D}^2\theta|^2
	- \frac{1}{\theta}\mathrm{D}^2\theta:(\na\theta\otimes\na\theta) 
	+ \frac{|\na\theta|^4}{\theta^2}\bigg)\dx = J_1+\cdots+J_7.
\end{align*}
We deduce from Young's inequality and Assumption (A4) on $\kappa$ that
$$
  J_1\ge \frac{c_w}{2\tau}\int_\Omega\rho(\theta^2-\bar{\theta}^2)\dx, \quad
	J_2 \ge c_\kappa\int_\Omega(1+\theta^2)|\na\theta|^2\dx.
$$
Furthermore, $J_3\ge 0$. Definition \eqref{3.defAB} of $B_i$ and $A_{ij}$ as well as
the bound $\rho_j\le \rho^*$ show that
\begin{align*}
  J_4 &= \theta\sum_{i,j=1}^n\frac{A_{ij}}{m_im_j\rho_i}\na\rho_i\cdot\na\theta\dx
	= \theta\sum_{i,j=1}^n\frac{M_{ij}^{BD}}{m_im_j}\frac{\sqrt{\rho_j}}{\sqrt{\rho_i}}
	\na\rho_i\cdot\na\theta\dx \\
	&\ge -\frac{c_\kappa}{2}\int_\Omega\theta^2|\na\theta|^2\dx 
	- C\int_\Omega\sum_{i=1}^n|\na\sqrt{\rho_i}|^2\dx.
\end{align*}
The integrals $J_5$ are $J_6$ are bounded from below since
$$  
  J_5 \ge -\frac{\lambda}{4}\int_{\pa\Omega}\theta_0^2\ds 
	\ge -C(\lambda,\pa\Omega,\theta_0), 
$$
and the dominant term in $J_6$ is $\theta^2\log\theta$, which is bounded from below
by a negative constant. Finally, $J_7$ is nonnegative:
$$
  J_7 = \frac{\eps}{2}\int_\Omega\bigg(|\mathrm{D}^2\theta|^2
	+ \frac{|\na\theta|^4}{\theta^2} + \bigg|\mathrm{D}^2\theta
	- \frac{1}{\theta}\na\theta\otimes\na\theta\bigg|^2\bigg)\dx \ge 0.
$$
Collecting these estimates finishes the proof.
\end{proof}


\subsection{Uniform estimates}\label{sec.unif}

Let $(w_1^k,\ldots,w_{n-1}^k,w^k)$ be a solution to the approximate scheme
\eqref{3.approx1}--\eqref{3.approx2} with $(w_1^{k-1},$ $\ldots,w_{n-1}^{k-1},w^{k-1})
=(\bar{w}_1,\ldots,\bar{w}_{n-1},\bar{w})$. We set $\theta^k=\exp(w^k)$ and
$\rho_i^k=\rho_i(w^k)$ determined from Lemma \ref{lem.inv}. Furthermore, we set
$E^k=c_w\rho\theta^k$, recalling that $\rho=\sum_{i=1}^n\rho_i^0$. We introduce the
piecewise constant in time functions
\begin{align*}
  & \rho_i^{(\tau)}(x,t) = \rho_i^k(x), \quad 
	q_i^{(\tau)} = \frac{1}{m_i}\log\frac{\rho_i^k}{m_i}-c_w(\log\theta^k-1)\quad
	\mbox{for }i=1,\ldots,n, \\
	& \theta^{(\tau)}(x,t) = \theta^k(x), \quad E^{(\tau)}(x,t) = E^k(x), \quad
	w_i^{(\tau)}(x,t) = w_i^k(x)\quad\mbox{for }i=1,\ldots,n-1,
\end{align*}
where $x\in\Omega$, $t\in((k-1)\tau,k\tau]$, and $k=1,\ldots,N$. At time $t=0$, we
set $\rho_i^{(\tau)}(0)=\rho_i^0$ and $\theta^{(\tau)}(0)=\theta^0$.
Furthermore, we introduce the shift operator $(\sigma_\tau\rho_i^{(\tau)})(x,t)
= \rho_i^{k-1}(x)$ if $t\in((k-1)\tau,k\tau]$. Then $(\bm{\rho}^{(\tau)},\theta^{(\tau)})$
solves
\begin{align}\label{3.weakrho}
  0 &= \frac{1}{\tau}\int_0^T\int_\Omega(\rho_i^{(\tau)}-\sigma_\tau\rho_i^{(\tau)})\phi_i
	\dx\dt + \eps\int_0^T\int_\Omega\big(\mathrm{D}^2w_i^{(\tau)}:\mathrm{D}^2\phi_i
	+ w_i^{(\tau)}\phi_i\big)\dx\dt \\
	&\phantom{xx}{}+ \int_0^T\int_\Omega\bigg(\sum_{j=1}^{n-1}A_{ij}(\bm{w}^{(\tau)})
	\na w_j^{(\tau)} + e^{-w^{(\tau)}}B_i(\bm{w}^{(\tau)})\na w^{(\tau)}\bigg)
	\cdot\na\phi_i\dx\dt, \nonumber \\
	0 &= \frac{1}{\tau}\int_0^T\int_\Omega(E^{(\tau)}-\sigma_\tau E^{(\tau)})\phi_0\dx\dt
	+ \int_0^T\int_\Omega\kappa(\theta^{(\tau)})\na\theta^{(\tau)}
	\cdot\na\phi_0\dx\dt \label{3.weaktemp} \\
	&\phantom{xx}{}+ \int_0^T\int_\Omega\sum_{i=1}^{n-1}B_j(\bm{w}^{(\tau)})
	\na w_i^{(\tau)}\cdot\na\phi_0\dx
	- \lambda\int_0^T\int_{\pa\Omega}(\theta_0-\theta^{(\tau)})\phi_0\ds\dt \nonumber \\
	&\phantom{xx}{}+ \int_0^T\int_\Omega\sum_{i,j=1}^n\frac{A_{ij}(\bm{w}^{(\tau)})}{m_im_j}
	\na\theta^{(\tau)}\cdot\na\phi_0\dx\dt \nonumber \\
	&\phantom{xx}{}
	+ \eps\int_0^T\int_\Omega(\theta_0+\theta^{(\tau)})(\log\theta^{(\tau)}-\log\theta_0)
	\phi_0\dx\dt \nonumber \\
	&\phantom{xx}{}+ \eps\int_0^T\int_\Omega\theta^{(\tau)}\big(\mathrm{D}^2\log\theta^{(\tau)}
	:\mathrm{D}^2\phi_0 + |\na\log\theta^{(\tau)}|^2\na\log\theta^{(\tau)}\cdot\na\phi_0\bigg)
	\dx\dt. \nonumber 
\end{align}
The discrete entropy inequality in Lemma \ref{lem.dei} and the temperature estimates
in Lemma \ref{lem.temp} yield, after summation over $k=1,\ldots,N$,
\begin{align}\label{3.unif.rho}
  & \sup_{0<t<T}\int_\Omega\bigg(h(\rho_1^{(\tau)}(t),\ldots,\rho_{n-1}^{(\tau)}(t),
	\theta^{(\tau)}(t)) + \frac{c_w}{\theta_0}\rho\theta^{(\tau)}(t)\bigg)\dx \\
	&\phantom{xxxx}{}
	+ \int_0^T\int_\Omega\big(|\na\log\theta^{(\tau)}|^2 + |\na\theta^{(\tau)}|^2\big)\dx\dt 
	\nonumber \\
	&\phantom{xxxx}{}+ \eps C\int_0^T\big(\|\bm{w}^{(\tau)}\|_{H^2(\Omega)}^2
	+ \|\na w^{(\tau)}\|_{L^4(\Omega)}^4\big)\dt\nonumber  \\
	&\phantom{xxxx}{}
	+ \int_0^T\int_\Omega\sum_{i=1}^n\frac{\mu}{m_i^2}\big|2\na(\rho_i^{(\tau)})^{1/2}
	+ (\rho_i^{(\tau)})^{1/2}\na\log\theta^{(\tau)}\big|^2\dx\dt\nonumber  \\
	&\phantom{xx}{}\le \int_\Omega\big(h(\rho_1^0,\ldots,\rho_{n-1}^0,\theta^0)
	+ c_w\rho\theta^0\big)\dx, \nonumber \\
	& c_w\sup_{0<t<T}\int_\Omega\rho(\theta^{(\tau)})^2\dx
	+ c_\kappa\int_0^T\int_\Omega(1+(\theta^{(\tau)})^2)|\na\theta^{(\tau)}|^2\dx\dt 
	\label{3.unif.temp} \\
	&\phantom{xx}{}\le C(T) + C\int_0^T\int_\Omega\sum_{i=1}^n|\na(\rho_i^{(\tau)})^{1/2}|^2\dx\dt
	+ \frac{c_w}{2}\int_\Omega\rho(\theta^0)^2\dx. \nonumber 
\end{align}

\begin{lemma}
There exists $C>0$ not depending on $(\eps,\tau)$ such that
\begin{align}
  \|\bm{\rho}^{(\tau)}\|_{L^\infty(\Omega_T)} 
	+ \|\theta^{(\tau)}\|_{L^\infty(0,T;L^1(\Omega))} &\le C, \label{est.rho} \\
	\|\log\theta^{(\tau)}\|_{L^2(0,T;H^1(\Omega))} 
	+ \|\theta^{(\tau)}\|_{L^2(0,T;H^1(\Omega))} &\le C, \label{est.theta} \\
	\eps^{1/2}\|\bm{w}^{(\tau)}\|_{L^2(0,T;H^2(\Omega))} 
	+ \eps^{1/4}\|\na w^{(\tau)}\|_{L^4(\Omega_T)} &\le C, \label{est.w}
\end{align}
\end{lemma}

\begin{proof}
Estimates \eqref{est.rho} and \eqref{est.w} are an immediate consequence of
\eqref{3.unif.rho} and $\rho\ge\rho_*>0$. Bound \eqref{3.unif.rho} also shows that 
$\sup_{(0,T)}\int_\Omega(-\log\theta^{(\tau)}+\theta^{(\tau)})\dx$
is uniformly bounded from above. 
Thus, $\log\theta^{(\tau)}$ is uniformly bounded in $L^\infty(0,T;L^1(\Omega))$.
Then the uniform bounds for $\na\log\theta^{(\tau)}$ and $\na\theta^{(\tau)}$
as well as the Poincar\'e--Wirtinger inequality yield bounds for
$\log\theta^{(\tau)}$ and $\theta^{(\tau)}$ in $L^2(\Omega_T)$, proving \eqref{est.theta}.
\end{proof}

\begin{lemma}
There exists $C>0$ not depending on $(\eps,\tau)$ such that for $i=1,\ldots,n$,
\begin{align}\label{est.rho2}
  \|(\rho_i^{(\tau)})^{1/2}\|_{L^2(0,T;H^1(\Omega))} 
	+ \|\rho_i^{(\tau)}\|_{L^2(0,T;H^1(\Omega))} &\le C, \\
	\|\theta^{(\tau)}\|_{L^\infty(0,T;L^2(\Omega))}
	+ \|(\theta^{(\tau)})^2\|_{L^2(0,T;H^1(\Omega))} 
	+ \|\theta^{(\tau)}\|_{L^{16/3}(\Omega_T)}&\le C. \label{est.theta2}
\end{align}
\end{lemma}

\begin{proof}
We infer from \eqref{3.unif.rho} that 
\begin{align*}
  \int_0^T\int_\Omega|\na(\rho_i^{(\tau)})^{1/2}|^2\dx\dt
	&\le C\int_0^T\int_\Omega\big|2\na(\rho_i^{(\tau)})^{1/2}|^2
	+ (\rho_i^{(\tau)})^{1/2}\na\log\theta^{(\tau)}\big|^2\dx\dt \\
	&\phantom{xx}{}+ C\int_0^T\int_\Omega|\na\log\theta^{(\tau)}|^2\dx\dt \le C,
\end{align*}
and the $L^\infty(\Omega_T)$ bound \eqref{est.rho} gives for $i=1,\ldots,n$,
\begin{equation*}
  \|\rho_i^{(\tau)}\|_{L^2(0,T;H^1(\Omega))}
	\le 2\|\rho_i^{(\tau)}\|_{L^\infty(\Omega_T)}^{1/2}
	\|\na(\rho_i^{(\tau)})^{1/2}\|_{L^2(\Omega_T)} + \|\rho_i^{(\tau)}\|_{L^2(\Omega_T)}
	\le C.
\end{equation*}
Therefore, the right-hand side of \eqref{3.unif.temp} is uniformly bounded,
which proves the first two estimates in \eqref{est.theta2}. 
The remaining one is a consequence of the Gagliardo--Nirenberg inequality
with $\eta=3/4$:
\begin{align*}
  \|(\theta^{(\tau)})^2\|_{L^{8/3}(\Omega_T)}^{8/3}
	&\le C\int_0^T\|(\theta^{(\tau)})^2\|_{H^1(\Omega)}^{8\eta/3}
	\|(\theta^{(\tau)})^2\|_{L^1(\Omega)}^{8(1-\eta)/3}\dt \\
	&\le \|\theta^{(\tau)}\|_{L^\infty(0,T;L^2(\Omega))}^{4/3}
	\int_0^T\|(\theta^{(\tau)})^2\|_{H^1(\Omega)}^2\dt \le C.
\end{align*}
This finishes the proof.
\end{proof}

The following lemma can be proved as in \cite[Lemma 9]{HeJu21}.

\begin{lemma}
There exists $C>0$ not depending on $(\eps,\tau)$ such that
\begin{equation}\label{est.time}
  \|\rho_i^{(\tau)}-\sigma_\tau\rho_i^{(\tau)}\|_{L^2(0,T;H^2(\Omega)^*)}
	+ \|\theta^{(\tau)}-\sigma_\tau\theta^{(\tau)}\|_{L^{16/15}(0,T;W^{2,16}(\Omega)^*)}
	\le C\tau.
\end{equation}
\end{lemma}


\subsection{The limit $(\eps,\tau)$}

The bounds \eqref{est.theta}, \eqref{est.rho2}, and \eqref{est.time} 
allow us to apply the Aubin--Lions lemma in the version of \cite{DrJu12}. There exist
subsequences, which are not relabeled, such that as $(\eps,\tau)\to 0$,
$$
  \rho_i^{(\tau)}\to\rho_i, \quad \theta^{(\tau)}\to\theta\quad\mbox{strongly in }
	L^2(\Omega_T),\ i=1,\ldots,n-1.
$$
The convergence also holds for $i=n$ since $\rho_n^{(\tau)}=1-\sum_{i=1}^{n-1}\rho_i^{(\tau)}$.
Thanks to the $L^\infty(\Omega_T)$ bound for
$\rho_i^{(\tau)}$ and the $L^{16/3}(\Omega_T)$ bound for $\theta^{(\tau)}$, we have
\begin{align*}
  \rho_i^{(\tau)}\to\rho_i &\quad\mbox{strongly in }L^r(\Omega_T)\mbox{ for all }
	r<\infty, \\
	\quad\theta^{(\tau)}\to\theta &\quad\mbox{strongly in }L^r(\Omega_T)
	\mbox{ for all }r<16/3.
\end{align*}

We claim that $\rho_i>0$ and $\theta>0$ a.e.\ in $\Omega_T$.
The positivity of $\rho_i$ is proved as in \cite[p.~16]{HeJu21}.
The strong convergence of $(\theta^{(\tau)})$ implies a.e.\ convergence and in particular
$\log\theta^{(\tau)}\to Z$ a.e. Thus, $\theta^{(\tau)}\to \exp(Z)$ a.e.
We conclude that $\theta=\exp(Z)>0$ a.e.\ in $\Omega_T$. 

It follows that $\log\theta\in L^2(\Omega_T)$ and estimate \eqref{est.theta} yields
\begin{equation}\label{4.nalog}
  \na\log\theta^{(\tau)}\rightharpoonup\na\log\theta\quad\mbox{weakly in }L^2(\Omega_T).
\end{equation}
Furthermore, in view of \eqref{est.theta}, \eqref{est.rho2}, and \eqref{est.time},
up to subsequences,
\begin{align*}
  \rho_i^{(\tau)}\rightharpoonup\rho_i, \quad \theta^{(\tau)}\rightharpoonup\theta
	&\quad\mbox{weakly in }L^2(0,T;H^1(\Omega)), \\
	\tau^{-1}(\rho_i^{(\tau)}-\sigma_\tau\rho_i^{(\tau)})\rightharpoonup\pa_t\rho_i
	&\quad\mbox{weakly in }L^2(0,T;H^2(\Omega)^*), \\
	\tau^{-1}(\theta^{(\tau)}-\sigma_\tau\theta^{(\tau)})\rightharpoonup\pa_t\rho_i
	&\quad\mbox{weakly in }L^{16/15}(0,T;W^{2,16}(\Omega)^*),
\end{align*}
and the bounds \eqref{est.w} show that
$$
  \eps\log\theta^{(\tau)}\to 0, \quad \eps w_i^{(\tau)}\to 0\quad\mbox{strongly in }
	L^2(0,T;H^2(\Omega)).
$$
The embedding $H^1(\Omega)\hookrightarrow L^2(\pa\Omega)$ is compact, giving
$\theta^{(\tau)}\to\theta$ strongly in $L^2(0,T;L^2(\pa\Omega))$.

These convergences are sufficient to pass to the limit $(\eps,\tau)\to 0$ in
\eqref{3.weakrho}--\eqref{3.weaktemp}, showing that $(\bm\rho,\theta)$ solves
the weak formulation \eqref{3.mass}--\eqref{3.energy}. We only detail the limits
in the terms $A_{ij}^{(\tau)}=A_{ij}(\bm{w}^{(\tau)})$ and 
$B_i^{(\tau)}=B_i(\bm{w}^{(\tau)})$. 
We know that $\na(\rho_i^{(\tau)})^{1/2}\rightharpoonup
\na\rho_i^{1/2}$ weakly in $L^2(\Omega_T)$ and
$$
  \frac{A_{ij}^{(\tau)}}{m_j(\rho_j^{(\tau)})^{1/2}}
	= M_{ij}^{BD}(\bm{\rho}^{(\tau)})\frac{(\rho_i^{(\tau)})^{1/2}}{m_j}
	\to M_{ij}^{BD}(\bm\rho)\frac{\rho_i^{1/2}}{m_j}
	= \frac{A_{ij}}{m_j\rho_j^{1/2}}
$$
strongly in $L^\gamma(\Omega_T)$ for all $\gamma<\infty$. Using \eqref{3.relA}
and \eqref{3.qi}, this implies that
\begin{align*}
  \sum_{j=1}^{n-1}A_{ij}^{(\tau)}\na w_j^{(\tau)}
	&= \sum_{j=1}^n \frac{A_{ij}^{(\tau)}}{m_j}\na\log\frac{\rho_i^{(\tau)}}{m_j}
	= 2\sum_{j=1}^{n}M^{BD}_{ij}(\bm{\rho}^{(\tau)})\frac{(\rho_i^{(\tau)})^{1/2}}{m_j}
	\na(\rho_j^{(\tau)})^{1/2} \\
	&\rightharpoonup
	2\sum_{j=1}^{n}M^{BD}_{ij}\frac{\rho_i^{1/2}}{m_j}\na\rho_j^{1/2}
	\quad\mbox{weakly in }L^s(\Omega_T),\ s<2.
\end{align*}
Since the sequence is bounded in $L^2(\Omega_T)$,
this convergence also holds in this space. Similarly, 
\begin{align*}
  & B_i^{(\tau)}e^{-w^{(\tau)}}\na w^{(\tau)}
	= \sum_{j=1}^n \frac{A_{ij}^{(\tau)}}{m_j}\na\log\theta^{(\tau)}
	\rightharpoonup \sum_{j=1}^n \frac{A_{ij}}{m_j}\na\log\theta
	\quad\mbox{weakly in }L^2(\Omega_T), \\
	& A_{ij}^{(\tau)}\na\theta^{(\tau)} = M^{BD}_{ij}(\bm{\rho}^{(\tau)})
	(\rho_i^{(\tau)}\rho_j^{(\tau)})^{1/2}\na\theta^{(\tau)}
	\rightharpoonup A_{ij}\na\theta\quad\mbox{weakly in }L^2(\Omega_T),
\end{align*}
and using $\theta^{(\tau)}\to \theta$ strongly in $L^r(\Omega_T)$ for $r<16/3$,
$$
  \sum_{i=1}^{n-1}B_i^{(\tau)}\na w_i^{(\tau)}
	= 2\sum_{i,j=1}^n\frac{M^{BD}_{ij}(\bm{\rho}^{(\tau)})}{m_im_j}
	\theta^{(\tau)}(\rho_j^{(\tau)})^{1/2}\na(\rho_i^{(\tau)})^{1/2}
	\rightharpoonup 2\sum_{i=1}^n \frac{B_i}{m_i\rho_i^{1/2}}\na\rho_i^{1/2}
$$
weakly in $L^s(\Omega_T)$ for $s<16/11$, and since the right-hand side lies in
$L^{16/11}(\Omega_T)$, this convergence also holds in $L^{16/11}(\Omega_T)$.

Next, we claim that $\rho_i(0)$ and $\theta(0)$ satisfy the
initial data. The time derivative of the linear interpolant
$$
  \widetilde\rho_i^{(\tau)}(t) = \rho_i^k - \frac{k\tau-t}{\tau}(\rho_i^k-\rho_i^{k-1})
	\quad\mbox{for }(k-1)\tau<t<k\tau
$$
is bounded since, because of \eqref{est.time},
$$
  \|\pa_t\widetilde\rho_i^{(\tau)}\|_{L^2(0,T;H^2(\Omega)^*)}
	\le \tau^{-1}\|\rho_i^{(\tau)}-\sigma_\tau\rho_i^{(\tau)}\|_{L^2(0,T;H^2(\Omega)^*)}
	\le C.
$$
Thus, $\widetilde\rho_i^{(\tau)}$ is uniformly bounded in
$H^1(0,T;H^2(\Omega)^*)\hookrightarrow C^0([0,T];H^2(\Omega)^*)$ and we conclude
for a subsequence that $\rho_i^0=\widetilde\rho_i^{(\tau)}(0)\rightharpoonup r_i$
weakly in $H^2(\Omega)^*$ for some $r_i\in H^2(\Omega)^*$. 
It follows that $r_i=\rho_i^0$. As $\widetilde\rho_i^{(\tau)}$
and $\rho_i^{(\tau)}$ converge to the same limit,
$$
  \|\widetilde\rho_i^{(\tau)}-\rho_i^{(\tau)}\|_{L^2(0,T;H^2(\Omega)^*)}
	\le \|\rho_i^{(\tau)}-\sigma_\tau\rho_i^{(\tau)}\|_{L^2(0,T;H^2(\Omega)^*)}
	\le C\tau\to 0,
$$
this shows that $\rho_i^0=r_i=\rho_i(0)$ in $H^2(\Omega)^*$. In an analogous way,
we verify that $\theta(0)=\theta^0$ in $W^{2,16}(\Omega)^*$. 

The initial data are satisfied in better spaces. Indeed, going back to 
\eqref{3.mass}--\eqref{3.energy}, the regularity of $\rho_i$ implies that
$\pa_t\rho_i\in L^2(0,T;H^1(\Omega))\cap H^1(0,T;H^1(\Omega)^*)\hookrightarrow
C^0([0,T];L^2(\Omega))$ and thus $\rho_i(0)=\rho_i^0$ in the sense of $L^2(\Omega)$.
The temperature satisfies $\theta\in L^\infty(0,T;L^2(\Omega))\cap 
C^0([0,T];W^{2,16}(\Omega)^*)$, which gives $\theta\in C_w^0([0,T];L^2(\Omega))$.
Consequently, $\theta(0)=\theta^0$ weakly in $L^2(\Omega)$.
Moreover, we deduce from $|\kappa\na\theta|\le C_\kappa(|\na\theta|+\theta|\na\theta^2|)
\in L^{16/11}(\Omega_T)$ that $\pa_t\theta \in L^{16/11}(0,T;W^{1,16/11}(\Omega)^*)$.
This completes the proof.


\section{Proof of Theorem \ref{thm.wsu}}\label{sec.wsu}

Let $(\bm\rho,\theta)$ be a weak solution and $(\bar{\bm\rho},\bar\theta)$ be a strong
solution to \eqref{1.mass}--\eqref{1.constit}. We introduce the entropy
$$
  H(\bm\rho(t),\theta(t)) = \int_\Omega\bigg(\sum_{i=1}^{n}\frac{\rho_i}{m_i}
	\bigg(\log\frac{\rho_i}{m_i}-1\bigg)
	- c_w\rho\log\theta\bigg)\dx.
$$

\begin{lemma}[Entropy equality for strong solutions]\label{lem.eis}
Let $(\bar{\bm{\rho}},\bar\theta)$ be a strong solution to \eqref{1.mass}--\eqref{1.constit}
(in the sense mentioned after Theorem \ref{thm.wsu}) with $\lambda=0$. Then
$$
  H(\bar{\bm{\rho}}(t),\bar\theta(t)) 
	+ \int_0^t\int_\Omega\frac{\kappa(\bar\theta)}{\bar\theta^2}|\na\bar\theta|^2\dx\ds
	+ \frac12\int_0^t\int_\Omega\sum_{i,j=1}^n b_{ij}\bar\rho_i\bar\rho_j
	|\bar{u}_i-\bar{u}_j|^2\dx\ds
	= H(\bar{\bm{\rho}}(0),\bar\theta(0)).
$$
\end{lemma}

\begin{proof}
We use \eqref{1.mass} and \eqref{1.energy} and integrate by parts to obtain
\begin{align*}
  \frac{\mathrm{d}H}{\dt} &= \int_\Omega\bigg(\sum_{i=1}^n
	\frac{\pa_t\bar\rho_i}{m_i}\log\frac{\bar\rho_i}{m_i} 
	- \frac{c_w}{\rho}\pa_t(\rho\bar\theta)\bigg)\dx \\
	&= \int_\Omega\bigg\{\sum_{i=1}^n\frac{\bar\rho_i\bar{u}_i}{m_i}\na\log\frac{\bar\rho_i}{m_i}
	+ \frac{\na\bar\theta}{\bar\theta^2}\bigg({-\bar\kappa\na\bar\theta} 
	+ \bar\theta\sum_{i=1}^n\frac{\bar\rho_i\bar{u}_i}{m_i}\bigg)\bigg\}dx \\
	&= -\int_\Omega\frac{\bar\kappa}{\bar\theta^2}|\na\bar\theta|^2\dx
	+ \int_\Omega\sum_{i=1}^n\frac{\bar{u}_i}{m_i}
	\cdot(\na\bar\rho_i+\bar\rho_i\na\log\bar\theta)\dx \\
	&= -\int_\Omega\frac{\bar\kappa}{\bar\theta^2}|\na\bar\theta|^2\dx
	+ \int_\Omega\sum_{i=1}^n\frac{1}{\bar\theta}\bar{u}_i\cdot \bar{d}_i\dx,
\end{align*}
where $\bar\kappa=\kappa(\bar\theta)$ and we used \eqref{2.di} in the last step. 
By the algebraic system \eqref{1.velo} and the symmetry of $(b_{ij})$,
\begin{equation}\label{5.uidi}
  \sum_{i=1}^n\frac{1}{\bar{\theta}}\bar{u}_i\cdot \bar{d}_i
	= -\sum_{i,j=1}^n b_{ij}\bar\rho_i\bar\rho_j(\bar{u}_i-\bar{u}_j)\cdot\bar{u}_i
	= -\frac12\sum_{i,j=1}^n b_{ij}\bar\rho_i\bar\rho_j|\bar{u}_i-\bar{u}_j|^2.
\end{equation}
This shows the claim.
\end{proof}

\begin{lemma}[Entropy inequality for weak solutions]\label{lem.eiw}
Let $(\bar{\bm{\rho}},\bar\theta)$ be a weak solution to \eqref{1.mass}--\eqref{1.constit}
with $\lambda=0$. Then
$$
  H(\bm\rho(t),\theta(t)) + \int_0^t\int_\Omega\frac{\kappa}{\theta^2}
	|\na\theta|^2\dx\ds
	+ \frac12\int_0^t\int_\Omega\sum_{i,j=1}^n 
	b_{ij}\rho_i\rho_j|u_i-u_j|^2\dx\ds	\le H(\bm\rho^0,\theta^0).
$$
\end{lemma}

\begin{proof}
Let $(\bm\rho^k,\theta^k)$ for $k=1,\ldots,N$ be a solution to the approximate problem
\eqref{3.approx1}--\eqref{3.approx2}, constructed in Section \ref{sec.LM}. 
According to \eqref{3.dei}, this solution satisfies 
\begin{align*}
  H(\bm\rho^k,\theta^k) &+ \tau\int_\Omega\kappa(\theta^k)|\nabla\log\theta^k|^2\dx \\
	& + \tau\int_\Omega\sum_{i,j=1}^n A_{ij}^k
	\na\bigg(q_i^k+\frac{w^k}{m_i}\bigg)\cdot\na\bigg(q_j^k+\frac{w^k}{m_j}\bigg)\dx
	\le H(\bm\rho^{k-1},\theta^{k-1}),
\end{align*}
where the superindex $k$ denotes the $k$th time step. By Corollary \ref{coro.flux}
as well as relations \eqref{3.ui} and \eqref{5.uidi},
\begin{align*}
  \sum_{i,j=1}^n A_{ij}^k &
	\na\bigg(q_i^k+\frac{w^k}{m_i}\bigg)\cdot\na\bigg(q_j^k+\frac{w^k}{m_j}\bigg)
	= \sum_{i,j=1}^n (M_{ij}^{BD})^k
	\frac{d_i^k}{\theta^k(\rho_i^k)^{1/2}}\cdot\frac{d_j^k}{\theta^k(\rho_j^k)^{1/2}} \\
	&= -\sum_{i=1}^n\frac{1}{\theta^k}d_i^k\cdot u_i^k
	= \frac12\sum_{i,j=1}^n b_{ij}\rho_i^k\rho_j^k|u_i^k-u_j^k|^2.
\end{align*}
Therefore,
$$
  H(\bm\rho^k,\theta^k) + \tau\int_\Omega\kappa(\theta^k)|\nabla\log\theta^k|^2\dx
	+ \frac{\tau}{2}\int_\Omega\sum_{i,j=1}^n b_{ij}\rho_i^k\rho_j^k|u_i^k-u_j^k|^2\dx
	\le H(\bm\rho^{k-1},\theta^{k-1}).
$$
We sum over $k=1,\ldots,j$ with $t\in((j-1)\tau,j\tau]$ 
and use the notation of Section \ref{sec.unif}:
\begin{align}\label{5.dei}
  H(\bm\rho^{(\tau)}(t),\theta^{(\tau)}(t)) 
	&+ \int_0^t\int_\Omega\kappa(\theta^{(\tau)})|\na\log\theta^{(\tau)}|^2\dx\ds \\
	&{}+ \frac12\int_0^t\int_\Omega\sum_{i,j=1}^n 
	b_{ij}\rho_i^{(\tau)}\rho_j^{(\tau)}|u_i^{(\tau)}-u_j^{(\tau)}|^2\dx\ds
	\le H(\bm\rho^0,\theta^0) \nonumber
\end{align}
for a.e.\ $t\in(0,T)$. 

It remains to pass to the limit $(\eps,\tau)\to 0$ in \eqref{5.dei}. 
We deduce from the strong convergence
of $(\bm\rho^{(\tau)})$ and $(\theta^{(\tau)})$ that
$$
  H(\bm\rho(t),\theta(t)) 
	\le \liminf_{(\eps,\tau)\to 0}H(\bm\rho^{(\tau)}(t),\theta^{(\tau)}(t)).
$$
We deduce from the strong convergence $\rho_i^{(\tau)}\to\rho_i$ in $L^q(\Omega_T)$ for any 
$q<\infty$ and the boundedness of $M_{ij}^{BD}$ that 
$M_{ij}^{BD}(\bm\rho^{(\tau)})\to M_{ij}^{BD}(\bm\rho)$ strongly in any $L^q(\Omega_T)$.
In view of the weak convergences $\na\log\theta^{(\tau)}\rightharpoonup\na\log\theta$
from \eqref{4.nalog} and $\na(\rho_i^{(\tau)})^{1/2}\rightharpoonup\na\rho_i^{1/2}$
from \eqref{est.rho2} weakly in $L^2(\Omega_T)$, we have
$$
  2\na(\rho_i^{(\tau)})^{1/2} + \rho_i^{(\tau)}\na\log\theta^{(\tau)}
	\rightharpoonup 2\na\rho_i^{1/2} + \rho_i\na\log\theta\quad\mbox{weakly in }L^2(\Omega_T).
$$
Hence, using \eqref{3.ui},
\begin{align*}
  (\rho_i^{(\tau)})^{1/2}u_i^{(\tau)}
	&= \sum_{j=1}^n M_{ij}^{BD}(\bm\rho^{(\tau)})\frac{1}{m_j}
	\big(2\na(\rho_i^{(\tau)})^{1/2} + \rho_i^{(\tau)}\na\log\theta^{(\tau)}\big) \\
	&\rightharpoonup \sum_{j=1}^n M_{ij}^{BD}(\bm\rho)\frac{1}{m_j}
	\big(2\na\rho_i^{1/2} + \rho_i\na\log\theta\big) = \rho_i^{1/2}u_i.
\end{align*}
weakly in $L^2(\Omega_T)$, where the last identity is the definition of $u_i$.
Then, taking into account the boundedness of $\rho_i^{(\tau)}$ in $L^\infty(\Omega_T)$,
for any $i,j=1,\ldots,n$,
$$
  (b_{ij}\rho_i^{(\tau)}\rho_j^{(\tau)})^{1/2}u_i^{(\tau)}
	\rightharpoonup (b_{ij}\rho_i\rho_j)^{1/2}u_i\quad\mbox{weakly in }L^2(\Omega_T).
$$ 
As the $L^2(\Omega_T)$ norm is weakly lower semicontinuous,
\begin{align*}
  \int_0^T\int_\Omega\sum_{i,j=1}^n b_{ij}\rho_i\rho_j|u_i-u_j|^2\dx\ds
	&\le \liminf_{(\eps,\tau)\to 0}\int_0^T\int_\Omega\sum_{i,j=1}^n
	\big|(b_{ij}\rho_i^{(\tau)}\rho_j^{(\tau)})^{1/2}(u_i^{(\tau)}-u_j^{(\tau)})\big|^2\dx\ds \\
	&= \liminf_{(\eps,\tau)\to 0}\int_0^T\int_\Omega\sum_{i,j=1}^n
	b_{ij}\rho_i^{(\tau)}\rho_j^{(\tau)}|u_i^{(\tau)}-u_j^{(\tau)}|^2\dx\ds.
\end{align*}
Finally, $\kappa(\theta^{(\tau)})^{1/2}\na\log\theta^{(\tau)}
\to\kappa(\theta)^{1/2}\na\log\theta$ weakly in $L^1(\Omega_T)$ and, because of
the uniform bounds, also in $L^2(\Omega_T)$. Hence,
$$
  \int_0^t\int_\Omega\frac{\kappa(\theta)}{\theta^2}|\na\theta|^2\dx\ds
	\le\liminf_{(\eps,\tau)\to 0}\int_0^t\int_\Omega
	\frac{\kappa(\theta^{(\tau)})}{(\theta^{(\tau)})^2}|\na\theta^{(\tau)}|^2\dx\ds.
$$
Thus, applying the limit inferior $(\eps,\tau)\to 0$ to both sides of \eqref{5.dei}
yields the result.
\end{proof}

\begin{lemma}[Relative entropy inequality]\label{lem.rei}
Let the assumptions of Theorem \ref{thm.wsu} hold and let $\rho_i(0)=\bar\rho_i(0)$ 
for $i=1,\ldots,n$ and $\theta(0)=\bar\theta(0)$. Then
\begin{align}\label{5.rei}
  H(&(\bm\rho,\theta)(t)|(\bar{\bm\rho},\bar\theta)(t))
	+ \frac{\mu_M}{2}\int_0^t\int_\Omega\sum_{i=1}^n\rho_i|u_i-\bar{u}_i|^2\dx\ds \\
	&{}+ \frac{c_\kappa}{2}\int_0^t\int_\Omega|\na(\log\theta-\log\bar\theta)|^2\dx\ds
	\le C\int_0^t\int_\Omega\bigg(\sum_{j=1}^n(\rho_j-\bar\rho_j)^2 + (\theta-\bar\theta)^2
	\bigg)\dx\ds, \nonumber
\end{align}
where the relative entropy $H(\bm\rho,\theta|\bar{\bm\rho},\bar\theta)$ 
is defined in \eqref{1.re}.
\end{lemma}

\begin{proof} We use the test functions $\phi_i=m_i^{-1}\log(\bar\rho_i/m_i)-c_w\log\bar\theta$
and $\phi_0=-1/\bar\theta$ in the weak formulations satisfied by $\rho_i-\bar\rho_i$ and
$\rho(\theta-\bar\theta)$, respectively,
\begin{align*}
  \int_\Omega(\rho_i-\bar\rho_i)(t)\phi_i(t)\dx 
	&= \int_0^t\int_\Omega(\rho_i-\bar\rho_i)\pa_t\phi_i\dx\ds
	+ \int_0^t\int_\Omega(\rho_iu_i-\bar\rho_i\bar{u}_i)\cdot\na\phi_i \dx\ds, \\
	\int_\Omega c_w\rho(\theta-\bar\theta)(t)\phi_0(t)\dx
	&= \int_0^t\int_\Omega c_w\rho(\theta-\bar\theta)\pa_t\phi_0\dx\ds
	- \int_0^t\int_\Omega(\kappa\na\theta-\bar\kappa\na\bar\theta)\cdot\na\phi_0\dx\ds \\
	&\phantom{xx}{}
	+ \int_0^t\int_\Omega\sum_{j=1}^n(h_ju_j-\bar{h}_j\bar{u}_j)\cdot\na\phi_0\dx\ds,
\end{align*}
where $h_j=(c_w+1/m_j)\rho_j\theta$, $\bar{h}_j=(c_w+1/m_j)\bar\rho_j\bar\theta$, 
and $\kappa=\kappa(\theta)$, $\bar\kappa=\kappa(\bar\theta)$. 
Strictly speaking, we cannot use $\phi_i$ as a 
test function since $\log\bar\rho_i$ and $1/\bar\theta$ may be not integrable. 
However, we can use
a density argument similarly as in the proof of \cite[Lemma 8]{HJT22}. 
Then, summing over $i=1,\ldots,n$,
\begin{align*}
  \int_\Omega&\bigg\{\sum_{i=1}^n(\rho_i-\bar\rho_i)(t)\bigg(\frac{1}{m_i}
	\log\frac{\bar\rho_i}{m_i} - c_w\log\bar\theta\bigg)(t) - c_w\rho
	\frac{\theta-\bar\theta}{\bar\theta}(t)\bigg\}\dx \\
	&= \int_0^t\int_\Omega\bigg\{\sum_{i=1}^n\bigg((\rho_i-\bar\rho_i)
	\frac{\pa_t\bar\rho_i}{m_i\bar\rho_i} + (\rho_iu_i-\bar\rho_i\bar{u}_i)
	\cdot\frac{\na\bar\rho_i}{m_i\bar\rho_i}\bigg)
	+ c_w\rho(\theta-\bar\theta)\pa_t\bigg({-\frac{1}{\bar\theta}}\bigg)\bigg\}\dx\ds \\
	&\phantom{xx}{}- \int_0^t\int_\Omega(\kappa\na\theta-\bar\kappa\na\bar\theta)\cdot\na
	\bigg({-\frac{1}{\bar\theta}}\bigg)\dx\ds
	+ \int_0^t\int_\Omega\sum_{j=1}^n(h_ju_j-\bar{h}_j\bar{u}_j)\cdot\na
	\bigg({-\frac{1}{\bar\theta}}\bigg)\dx\ds.
\end{align*}
We subtract this identity and the entropy equality from Lemma \ref{lem.eis} for 
$(\bar{\bm\rho},\bar\theta)$ from the entropy inequality for $(\bm\rho,\theta)$
obtained in Lemma \ref{lem.eiw}
and insert equations \eqref{1.mass}--\eqref{1.energy} to replace the time derivatives
$\pa_t\bar\rho_i$ and $\pa_t({-1/\bar\theta})$. A computation shows that
\begin{align}\label{5.H}
  & H((\bm\rho,\theta)(t)|(\bar{\bm\rho},\bar\theta)(t)) \le K_1\cdots+K_5, \quad\mbox{where} \\
  & K_1 = -\int_0^t\int_\Omega\big(\kappa|\na\log\theta|^2 
	- \bar\kappa|\na\log\bar\theta|^2\big)\dx\ds
	+ \int_0^t\int_\Omega\bar\kappa\na\bar\theta\cdot\na\bigg(\frac{\theta}{\bar\theta^2}
	- \frac{1}{\bar\theta}\bigg)\dx\ds \nonumber \\
	&\phantom{xxxx}{}+ \int_0^t\int_\Omega(\kappa\na\theta-\bar\kappa\na\bar\theta)
	\cdot\na\bigg({-\frac{1}{\bar\theta}}\bigg)\dx\ds, \nonumber \\
	& K_2 = -\int_0^t\int_\Omega\sum_{i=1}^n\frac{\bar\rho_i\bar{u}_i}{m_i}\cdot
	\na\bigg(\frac{\rho_i}{\bar\rho_i}\bigg)\dx\ds
	- \int_0^t\int_\Omega\sum_{i=1}^n\frac{\na\bar\rho_i}{m_i\bar\rho_i}
	\cdot(\rho_iu_i-\bar\rho_i\bar{u}_i)\dx\ds, \nonumber \\
	& K_3 = -\int_0^t\int_\Omega\sum_{i=1}^n\bar{h}_i\bar{u}_i\cdot\na
	\bigg(\frac{\theta}{\bar\theta^2}	- \frac{1}{\bar\theta}\bigg)\dx\ds
	- \int_0^t\int_\Omega\sum_{i=1}^n(h_iu_i-\bar{h}_i\bar{u}_i)\cdot\na
	\bigg({-\frac{1}{\bar\theta}}\bigg)\dx\ds, \nonumber \\
	& K_4 = -\frac12\int_0^t\int_\Omega\sum_{i,j=1}^n b_{ij}\rho_i\rho_j
	|u_i-u_j|^2\dx\ds, \nonumber \\
	& K_5 = \frac12\int_0^t\int_\Omega\sum_{i,j=1}^n b_{ij}\bar\rho_i\bar\rho_j
	\big|\bar{u}_i-\bar{u}_j\big|^2\dx\ds. \nonumber 
\end{align}
The term $K_1$ can be rewritten as
\begin{align*}
  K_1 &= -\int_0^t\int_\Omega\frac{1}{\bar\theta}(\kappa\bar\theta-\bar\kappa\theta)
	\na(\log\theta-\log\bar\theta)\cdot\na\log\bar\theta\dx\ds \\
	&\phantom{xx}{}-\int_0^t\int_\Omega\kappa\big|\na(\log\theta-\log\bar\theta)\big|^2\dx\ds \\
	&\phantom{xx}{}+ \int_0^t\int_\Omega\frac{\theta-\bar\theta}{\bar\theta}\na\log\bar\theta
 	\cdot(\kappa\na\log\theta-\bar\kappa\na\log\bar\theta)\dx\ds
	=: K_{11}+K_{12}+K_{13}.
\end{align*}
The algebraic system \eqref{1.velo} with $d_i=\na(\rho_i\theta)/m_i$ can be formulated as
$$
  -m_i\sum_{j=1}^n b_{ij}\bar\rho_i\bar\rho_j(\bar{u}_i-\bar{u}_j) 
	- \bar\rho_i\na\log\bar\theta = \na\bar\rho_i.
$$
This allows us to rewrite $K_2$:
\begin{align*}
  K_2 & = \int_0^t\int_\Omega \sum_{i,j=1}^nb_{ij}\rho_i\rho_j(u_i-u_j)\cdot\bar{u}_i
	\dx\ds - \int_0^t\int_\Omega \sum_{i,j=1}^nb_{ij}\rho_i\bar{\rho}_j(\bar{u}_i-\bar{u}_j)
	\cdot\bar{u}_i\dx\ds \\
  &\phantom{xx}{} + \int_0^t\int_\Omega \sum_{i,j=1}^nb_{ij}\rho_i\bar{\rho}_j
	(\bar{u}_i-\bar{u}_j)\cdot u_i\dx\ds - \int_0^t\int_\Omega \sum_{i,j=1}^nb_{ij}
	\bar{\rho}_i\bar{\rho}_j(\bar{u}_i-\bar{u}_j)\cdot\bar{u}_i\dx\ds \\
  &\phantom{xx}{} + \int_0^t\int_\Omega\sum_{i=1}^n\frac{1}{m_i}\rho_i\nabla\log\theta
	\cdot\bar{u}_i\dx\ds - \int_0^t\int_\Omega\sum_{i=1}^n\frac{1}{m_i}\rho_i\nabla
	\log\bar{\theta}\cdot\bar{u}_i\dx\ds \\
  &\phantom{xx}{} + \int_0^t\int_\Omega\sum_{i=1}^n\frac{1}{m_i}\rho_i\nabla\log\bar{\theta}
	\cdot u_i\dx\ds - \int_0^t\int_\Omega\sum_{i=1}^n\frac{1}{m_i}\bar{\rho}_i\nabla
	\log\bar{\theta}\cdot\bar{u}_i\dx\ds \\
	&=: K_{21}+\cdots+K_{28}.
\end{align*}
Furthermore, it follows from $h_i=(c_w+1/m_i)\rho_i\theta$ and 
$\sum_{i=1}^n\rho_iu_i=\sum_{i=1}^n\bar\rho_i\bar{u}_i=0$ that
\begin{align*}
  K_3 &= -\int_0^t\int_\Omega\sum_{i=1}^n\bar{h}_i\bar{u}_i\cdot
	\na\bigg(\frac{\theta}{\bar\theta^2}\bigg)\dx\ds
	- \int_0^t\int_\Omega\sum_{i=1}^n h_iu_i\cdot\na\bigg({-\frac{1}{\bar\theta}}\bigg)\dx\ds \\
	&= -\int_0^t\int_\Omega\sum_{i=1}^n\bigg(c_w+\frac{1}{m_i}\bigg)\bar\theta\bar\rho_i
	\bar{u}_i\cdot\na\bigg(\frac{\theta}{\bar\theta^2}\bigg)\dx\ds \\
	&\phantom{xx}{}- \int_0^t\int_\Omega\sum_{i=1}^n\bigg(c_w+\frac{1}{m_i}\bigg)\theta\rho_i
	u_i\cdot\na\bigg({-\frac{1}{\bar\theta}}\bigg)\dx\ds \\
	&= -\int_0^t\int_\Omega\sum_{i=1}^n\frac{\bar\rho_i\bar\theta}{m_i}
	\bar{u}_i\cdot\na\bigg(\frac{\theta}{\bar\theta^2}\bigg)\dx\ds
	- \int_0^t\int_\Omega\sum_{i=1}^n\frac{\rho_i\theta}{m_i}u_i\cdot
	\na\bigg({-\frac{1}{\bar\theta}}\bigg)\dx\ds \\
  &= -\int_0^T\int_\Omega\sum_{i=1}^n\frac{\bar{\rho}_i\bar{u}_i}{m_i\bar\theta}
	\cdot\na\theta\dx\ds+2\int_0^T\int_\Omega\sum_{i=1}^n\frac{\bar{\rho}_i\bar{u}_i
	\theta}{m_i\bar\theta^2}\cdot\nabla\bar\theta\dx\ds \\
  &\phantom{xx}-\int_0^T\int_\Omega\sum_{i=1}^n\frac{\rho_iu_i\theta}{m_i\bar\theta^2}
	\cdot\na\bar\theta\dx\ds.
\end{align*}
We reformulate $K_4$ as
\begin{align*}
  K_4 &= -\frac12\int_0^t\int_\Omega\sum_{i,j=1}^n b_{ij}\rho_i\rho_j
	\big|(u_i-\bar{u}_i)-(u_j-\bar{u}_j)\big|^2\dx\ds \\
	&\phantom{xx}{}+ \frac12\int_0^t\int_\Omega\sum_{i,j=1}^n 
	b_{ij}\rho_i\rho_j|\bar{u}_i-\bar{u}_j|^2\dx\ds \\
	&\phantom{xx}{}- \int_0^t\int_\Omega\sum_{i,j=1}^n b_{ij}\rho_i\rho_j
	(u_i-u_j)\cdot(\bar{u}_i-\bar{u}_j)\dx\ds
	=: K_{41}+K_{42}+K_{43}.
\end{align*}
A long but straightforward computation shows that 
\begin{align*}
    K_{21}&+K_{22}+ K_{23}+K_{24}+K_{42}+K_{43}+K_5 \\
    & = -\int_0^T\int_\Omega\sum_{i,j=1}^nb_{ij}\rho_i(\rho_j-\bar{\rho}_j)(u_i-\bar{u}_i)
		\cdot(\bar{u}_i-\bar{u}_j)\dx\ds=:L_1
\end{align*} 
and 
\begin{align*}         
    K_{25} &+K_{26} + K_{27}+K_{28}+K_3=\int_0^T\int_\Omega\sum_{i=1}^n\frac{1}{m_i}
		(\rho_i-\bar{\rho}_i)(\na\log\theta-\na\log\bar\theta)\cdot\bar{u}_i\dx\ds \\
    &\phantom{xx}{} + \int_0^T\int_\Omega\sum_{i=1}^n\frac{1}{m_i}\bar\rho_i\bar u_i\cdot
		(\na\log\theta-\na\log\bar\theta)\bigg(1-\frac{\theta}{\bar\theta}\bigg)\dx\ds \\
    &\phantom{xx}{} + \int_0^T\int_\Omega\sum_{i=1}^n\frac{1}{m_i}\rho_i(u_i-\bar u_i)
		\cdot\nabla\log\bar\theta\bigg(1-\frac{\theta}{\bar\theta}\bigg)\dx\ds \\
    &\phantom{xx}{} + \int_0^T\int_\Omega\sum_{i=1}^n\frac{1}{m_i}(\rho_i-\bar\rho_i)\bar u_i
		\cdot\nabla\log\bar\theta\bigg(1-\frac{\theta}{\bar\theta}\bigg)\dx\ds \\
	&=: L_2+L_3+L_4+L_5.
\end{align*}
Inserting these expressions into \eqref{5.H}, putting $K_{12}$ on the left-hand side, 
and rearranging the terms, we find that
\begin{align}\label{5.H2}
  & H((\bm\rho,\theta)(t)|(\bar{\bm\rho},\bar\theta)(t)) 
	+ \frac12\int_0^t\int_\Omega\sum_{i,j=1}^n b_{ij}\rho_i\rho_j
	\big|(u_i-\bar{u}_i)-(u_j-\bar{u}_j)\big|^2\dx\ds \\
	&\phantom{xx}{}+ \int_0^t\int_\Omega\kappa|\na(\log\theta-\log\bar\theta)|^2\dx\ds 
	\leq K_{11}+K_{13}+L_1+\cdots+L_5. \nonumber 
\end{align}

The second term on the left-hand side can be bounded from below. Indeed, it follows from
the symmetry of $(b_{ij})$, definition \eqref{3.M} of $M_{ij}$, and the positive
definiteness \eqref{3.posdefM} of $M$ on $L$ that
\begin{align*}
  \frac12\sum_{i,j=1}^n& b_{ij}\rho_i\rho_j\big|(u_i-\bar{u}_i)-(u_j-\bar{u}_j)\big|^2 \\
	&= \sum_{i=1}^n\bigg(\sum_{j=1,\,j\neq i}^n b_{ij}\rho_j\bigg)\rho_i|u_i-\bar{u}_i|^2
	- \sum_{i,j=1,\,i\neq j}^n b_{ij}\rho_i\rho_j(u_i-\bar{u}_i)\cdot(u_j-\bar{u}_j) \\
	&= \sum_{i,j=1}^n M_{ij}\sqrt{\rho_i}(u_i-\bar{u}_i)\cdot\sqrt{\rho_j}(u_j-\bar{u}_j)
	\ge \mu_M|P_L \bm{Y}|^2,
\end{align*}
where $Y_j=\sqrt{\rho_j}(u_j-\bar{u}_j)$. The norm of the projection is computed
according to
\begin{align*}
  |P_L\bm{Y}|^2 &= |\bm{Y}|^2 - |P_{L^\perp}\bm{Y}|^2 = \sum_{i=1}^n\rho_i|u_i-\bar{u}_i|^2
	- \sum_{i=1}^n\frac{\rho_i}{\rho^2}\bigg|\sum_{j=1}^n\rho_j(u_j-\bar{u}_j)\bigg|^2 \\
	&= \sum_{i=1}^n\rho_i|u_i-\bar{u}_i|^2 
	- \frac{1}{\rho}\bigg|\sum_{j=1}^n(\rho_j-\bar\rho_j)\bar{u}_j\bigg|^2
	\ge \sum_{i=1}^n\rho_i|u_i-\bar{u}_i|^2 - C_1\sum_{j=1}^n(\rho_j-\bar\rho_j)^2,
\end{align*}
where we used $\sum_{i=1}^n\rho_iu_i=0$ in the third equality, and $C_1>0$ depends on
$\rho_*$ and the $L^\infty(\Omega_T)$ norms of $\bar{u}_j$, $j=1,\ldots,n$.
Consequently,
\begin{align}\label{5.left}
  \frac12\int_0^t&\int_\Omega\sum_{i,j=1}^n b_{ij}\rho_i\rho_j
	\big|(u_i-\bar{u}_i)-(u_j-\bar{u}_j)\big|^2\dx\ds \\
	&\ge \mu_M\int_0^t\int_\Omega\sum_{i=1}^n\rho_i|u_i-\bar{u}_i|^2\dx\ds
	- C_2\int_0^t\int_\Omega\sum_{j=1}^n(\rho_j-\bar\rho_j)^2\dx\ds. \nonumber
\end{align}

We turn to the estimation of the terms on the right-hand side of \eqref{5.H2}.
By the Lipschitz continuity of $\kappa$ and Young's inequality, $K_{11}$ is estimated as
\begin{align*}
  K_{11} &= -\int_0^t\int_\Omega\frac{1}{\bar\theta}\big(\kappa(\bar\theta-\theta)
  + (\kappa-\bar\kappa)\theta\bigg)\na\bar\theta\cdot\na(\log\theta-\log\bar\theta)\dx\ds \\
	&\le \frac{c_\kappa}{8}\int_0^t\int_\Omega|\na(\log\theta-\log\bar\theta)|^2\dx\ds
	+ C_3\int_0^t\int_\Omega(\theta-\bar\theta)^2\dx\ds,
\end{align*}
and $C_3>0$ depends on $c_\kappa$ (see Assumption (A4)), and the $L^\infty(\Omega_T)$ norms
of $\theta$ and $\na\log\bar\theta$. A similar estimate shows that
\begin{align*}
  K_{13} &= -\int_0^t\int_\Omega\frac{\theta-\bar\theta}{\bar\theta}\big(
	\kappa\na(\log\theta-\log\bar\theta) + (\kappa-\bar\kappa)\na\log\bar\theta\big)
	\cdot\na\log\bar\theta\dx\ds \\
	&\le \frac{c_\kappa}{8}\int_0^t\int_\Omega|\na(\log\theta-\log\bar\theta)|^2\dx\ds
	+ C_4\int_0^t\int_\Omega(\theta-\bar\theta)^2\dx\ds, \\
	L_2 &\le \frac{c_\kappa}{8}\int_0^t\int_\Omega|\na(\log\theta-\log\bar\theta)|^2\dx\ds
	+ C_5\int_0^t\int_\Omega\sum_{i=1}^n(\rho_i-\bar\rho_i)^2\dx\ds, \\
	L_3 &\le \frac{c_\kappa}{8}\int_0^t\int_\Omega|\na(\log\theta-\log\bar\theta)|^2\dx\ds
	+ C_6\int_0^t\int_\Omega(\theta-\bar\theta)^2\dx\ds,
\end{align*}
observing that $C_4$ depends on $c_\kappa$, $\delta$ and the $L^\infty(\Omega_T)$ 
norms of $\theta$, $\nabla\log\bar\theta$, and $\bar{u}_i$,
$C_5$ depends on the $L^\infty(\Omega_T)$ norms of $\bar{u}_i$,  
and $C_6$ depends on $c_\kappa$, $\rho^*$, $\delta$, and the $L^\infty(\Omega_T)$ norms of 
$\bar{u}_i$ ($i=1,\ldots,n$). Moreover, by Young's inequality again,
\begin{align*}
  L_1 &\le \frac{\mu_M}{4}\int_0^t\int_\Omega\sum_{i=1}^n\rho_i|u_i-\bar{u}_i|^2\dx\ds
	+ C_7\int_0^t\int_\Omega\sum_{i=1}^n(\rho_i-\bar\rho_i)^2\dx\ds, \\
	L_4 &\le \frac{\mu_M}{4}\int_0^t\int_\Omega\sum_{i=1}^n\rho_i|u_i-\bar{u}_i|^2\dx\ds
	+ C_8\int_0^t\int_\Omega(\theta-\bar\theta)^2
	\dx\ds,
\end{align*}
where $C_7$ depends on $\rho^*$, $\mu_M$, and the $L^\infty(\Omega_T)$ norms of 
$\bar{u}_i$ ($i=1,\ldots,n$), while $C_8$ depends on $\delta$, $\rho^*$, and the 
$L^\infty(\Omega_T)$ norm of $\na\log\bar\theta$.
Finally, 
\begin{align*}
    L_5\leq C_9\int_0^T\int_\Omega\sum_{i=1}^n(\rho_i-\bar\rho_i)^2\dx\ds
		+ C_{10}\int_0^T\int_\Omega(\theta-\bar\theta)^2\dx\ds,
\end{align*} 
where $C_9>0$ depends on the $L^\infty(\Omega_T)$ norms of $\bar{u}_i$ ($i=1,\ldots,n$),
and $C_{10}$ depends on $\delta$ and the $L^\infty(\Omega_T)$ norm of $\na\log\bar\theta$.

Summarizing the previous estimations, we infer from \eqref{5.H2}, \eqref{5.left}, 
and the lower bound for $\kappa$ (see Assumption (A4)) the conclusion.
\end{proof}

It remains to estimate the right-hand side of \eqref{5.rei} in terms of the relative entropy.
For this, we observe that, by \cite[Lemma 16]{HJT22},
$$
  \int_\Omega\sum_{i=1}^n\frac{1}{m_i}\bigg(\rho_i\log\frac{\rho_i}{\bar\rho_i}
	- (\rho_i-\bar\rho_i)\bigg)\dx \ge C\int_\Omega\sum_{i=1}^n(\rho_i-\bar\rho_i)^2\dx.
$$
Furthermore, for all functions $f\in C^1(\R)$ with $f'(1)=0$,
\begin{align*}
  f(s)-f(1) &= (s-1)\int_0^1 f'(\sigma(s-1)+1)\mathrm{d}\sigma
	= (s-1)\int_0^1 f'(\tau(s-1)+1)\big|_{\tau=0}^\sigma\mathrm{d}\sigma \\
	&= (s-1)^2\int_0^1\int_0^\sigma f''(\tau(s-1)+1)\mathrm{d}\tau\mathrm{d}\sigma.
\end{align*}
This yields, choosing $f(s)=-\log s+s-1$ and $s=\theta/\bar\theta$,
$$
  \int_\Omega c_w\rho\bigg(-\log\frac{\theta}{\bar\theta} + \frac{1}{\bar\theta}
	(\theta-\bar\theta)\bigg)\dx \ge \int_\Omega c_w\rho
	\frac{(\theta-\bar\theta)^2}{\max\{\theta,\bar\theta\}^2}\dx
	\ge C\int_\Omega(\theta-\bar\theta)^2\dx,
$$
where $C>0$ depends on the lower bound for $\bar\theta$ in $\Omega_T$. 
By definition of the relative entropy, we conclude from Lemma \ref{lem.rei} that
\begin{align*}
  H&((\bm\rho,\theta)(t)|(\bar{\bm\rho},\bar\theta)(t)) 
	+ \frac{\mu_M}{2}\int_0^t\int_\Omega\sum_{i=1}^n\rho_i|u_i-\bar{u}_i|^2\dx\ds \\
	&{}+ \frac{c_\kappa}{2}\int_0^t\int_\Omega|\na(\log\theta-\log\bar\theta)|^2\dx\ds
	\le C\int_0^t H(\bm\rho,\theta|\bar{\bm\rho},\bar\theta)\ds.
\end{align*}
Gronwall's lemma shows that $H((\bm\rho,\theta)(t)|(\bar{\bm\rho},\bar\theta)(t))=0$
and hence $\bm\rho(t)=\bar{\bm\rho}(t)$ and $\theta(t)=\bar\theta(t)=0$ in $\Omega$
for $t>0$. This finishes the proof.


\end{document}